\documentclass[oneside,11pt]{amsart}

\usepackage{hyperref}

\usepackage{amsmath, amssymb, amsthm, latexsym, enumerate, mathrsfs, textcomp, bbm}
\usepackage{graphicx}
\usepackage{amsmath}
\usepackage[T1]{fontenc}

\usepackage[small]{caption}

\usepackage{tikz-cd}
\usetikzlibrary{arrows,chains,matrix,positioning,scopes}

\makeatletter
\tikzset{join/.code=\tikzset{after node path={%
\ifx\tikzchainprevious\pgfutil@empty\else(\tikzchainprevious)%
edge[every join]#1(\tikzchaincurrent)\fi}}}
\makeatother
\tikzset{>=stealth',every on chain/.append style={join},
         every join/.style={->}}
\tikzstyle{labeled}=[execute at begin node=$\scriptstyle,
   execute at end node=$]

\usepackage{amsfonts}
\usepackage{amssymb}
\usepackage{amsxtra}

\usepackage{ifthen}

\newcommand{\showcomments}{yes}

\newsavebox{\commentbox}
{\ifthenelse{\equal{\showcomments}{yes}}%
{\footnotemark
    \begin{lrbox}{\commentbox}
    \begin{minipage}[t]{1.25in}\raggedright\sffamily\upshape\tiny
    \footnotemark[\arabic{footnote}]}
{\begin{lrbox}{\commentbox}}}%
{\ifthenelse{\equal{\showcomments}{yes}}%
{\end{minipage}\end{lrbox}\marginpar{\usebox{\commentbox}}}
{\end{lrbox}}}

\theoremstyle{plain}
\newtheorem{theorem}{Theorem}[section]
\newtheorem{corollary}[theorem]{Corollary}
\newtheorem{lemma}[theorem]{Lemma}
\newtheorem{proposition}[theorem]{Proposition}
\newtheorem{problem}[theorem]{Problem}

\theoremstyle{definition}

\newtheorem*{remark}{Remark}

\newcommand{\Def}{\noindent {\bf Definition}: }

\newcommand{\spc}{\hspace{1 mm}}

\newcommand{\bbr}{\mathbb{R}}
\newcommand{\bbn}{\mathbb{N}}

\newcommand{\bbq}{\mathbb{Q}}

\newcommand{\bbp}{\mathbb{P}}
\newcommand{\bndry}{\partial}

\newcommand{\C}{\mathcal{C}}
\newcommand{\relbndry}{\bndry (\Gamma,\bbp)}
\newcommand{\pairs}{\Theta_2\bndry X}
\newcommand{\R}{\mathcal{R}}
\newcommand{\N}{\mathcal{N}}

\newcommand{\D}{\mathcal{D}}

\newcommand{\caygs}{\Upsilon(G,S)}
\newcommand{\caygt}{\Upsilon(G,T)}
\newcommand{\blockbndry}{\bndry (H,\bbq)}
\newcommand{\compbndry}{\bndry (H,\bbq)}
\newcommand{\relbndryG}{\bndry (G,\bbp)}
\newcommand{\A}{\mathcal{A}}
\newcommand{\B}{\mathcal{B}}

\newcommand{\U}{\mathcal{U}}
\newcommand{\sP}{\mathcal{P}}

\DeclareMathOperator{\CAT}{CAT}

\DeclareMathOperator{\stab}{Stab}
\DeclareMathOperator{\Stab}{Stab}
\DeclareMathOperator{\Hull}{Hull}

\DeclareMathOperator{\orb}{Orb}
\DeclareMathOperator{\Ends}{Ends}
\DeclareMathOperator{\Orb}{Orb}
\DeclareMathOperator{\Fr}{Fr}
\DeclareMathOperator{\Int}{Int}
\DeclareMathOperator{\im}{im}
\DeclareMathOperator{\Sat}{Sat}

\DeclareMathOperator{\Jn}{Join}

\DeclareMathOperator{\Jump}{Jump}

\newcommand{\gJn}{\Jn^{+}}

\newcommand{\set}[2]{\{\,{#1} \mid {#2} \,\}}
\newcommand{\bigset}[2]{ \bigl\{ \, {#1} \bigm| {#2} \, \bigr\} }

\begin{document}

\title[Local cut points and relatively hyperbolic groups]{%
 Local cut points and splittings
 of relatively hyperbolic groups}

\author[M.~Haulmark]{Matthew Haulmark}
\address{Department of Mathematics\\
1326 Stevenson Center\\
Vanderbilt University\\
Nashville, TN 37240 USA}
\email{m.haulmark@vanderbilt.edu}


\begin{abstract}
In this paper we show that the existence of a non-parabolic local cut point in the Bowditch boundary $\relbndryG$ of a relatively hyperbolic group $(G,\bbp)$ implies that $G$ splits over a $2$-ended subgroup. This theorem generalizes a theorem of Bowditch from the setting of hyperbolic groups to relatively hyperbolic groups. As a consequence we are able to generalize a theorem of Kapovich and Kleiner by classifying the homeomorphism type of $1$-dimensional Bowditch boundaries of relatively hyperbolic groups which satisfy certain properties, such as no splittings over $2$-ended subgroups and no peripheral splittings.

In order to prove the boundary classification result we require a notion of ends of a group which is more general than the standard notion. We show that if a finitely generated discrete group acts properly and cocompactly on two generalized Peano continua $X$ and $Y$, then $\Ends(X)$ is homeomorphic to $\Ends(Y)$. Thus we propose an alternate definition of $\Ends(G)$ which increases the class of spaces on which $G$ can act.
\end{abstract}


\date{\today}

\maketitle

\section{Introduction}
\label{sec:Introduction}

The notion of a group $G$ being hyperbolic relative to a class of subgroups $\bbp$ was introduced by Gromov \cite{Gro87} to generalize both word hyperbolic and geometrically finite Kleinian groups. The subgroups in the class $\bbp$ are called {\it peripheral subgroups}, and when $G$ is hyperbolic relative to $\bbp$ we often say $(G,\bbp)$ is relatively hyperbolic. Introduced by Bowditch \cite{Bow1} there is a boundary for relatively hyperbolic groups. The Bowditch boundary $\relbndryG$ generalizes the Gromov boundary of a word hyperbolic group and the limit set of a geometrically finite Kleinian group. The homeomorphism type of the Bowditch boundary is known to be a quasi-isometry invariant of the group \cite{Grof13} under modest hypotheses on the peripheral subgroups. Consequently, it is desirable to describe the topological features of the Bowditch boundary. Topological features of the boundary are closely related to algebraic properties of the group; in particular they are often related to splittings of the group as a fundamental group of a graph of groups.

A point $p$ in $\relbndryG$ is a {\it local cut point} if $\relbndryG\setminus\{p\}$ is disconnected, or $\relbndryG\setminus\{p\}$ is connected and has more than one end. For $1$-ended hyperbolic groups Bowditch \cite{Bow98} shows that the existence of a splitting over a $2$-ended subgroup (see Section \ref{subsec:splittings}) is equivalent to the existence of a local cut point in the Gromov boundary. As evidenced by the work of  Kapovich and Kleiner \cite{KK00}, this result has proved useful in classifying the homeomorphism type of $1$-dimensional boundaries of hyperbolic groups. Kapovich and Kleiner's result relies on the topological characterization of the Menger curve \cite{A58a, A58b}, which requires that the boundary has no local cut points. Because the existence or non-existence of $2$-ended splittings can be verified directly in many natural settings, Kapovich and Kleiner's results provide techniques for constructing examples of hyperbolic groups with Menger curve or Sierpinski carpet boundary. Obstructions to $2$-ended splittings are well understood for hyperbolic 3-manifold groups \cite{Myr93}, Coxeter groups \cite{MT09}, and random groups \cite{DGP11}.

Papasoglu-Swenson \cite{PS05,PS09}, and Groff \cite{Grof13} have extended Bowditch's results \cite{Bow98} from hyperbolic groups to $\CAT(0)$ and relatively hyperbolic groups respectively. Their results describe the relationship between $2$-ended splittings and cut pairs in the boundary. In particular, their results make no mention of local cut points. Guralnik \cite{Gur} has observed that many of Bowditch's local cut point results extend to relatively hyperbolic groups, provided that the Bowditch boundary has no global cut points. However, that assumption is quite restrictive. Bowditch has shown \cite{Bow01} that the Bowditch boundary often has many global cut points. Thus a general theorem relating local cut points in the Bowditch boundary to $2$-ended splittings was still missing from the literature. The primary result of this paper addresses the general setting with the following theorem that makes no assumption about the existence or non-existence of global cut points in the Bowditch boundary.

\begin{theorem}[Splitting Theorem]
\label{theorem:Splitting Theorem}
Let $(G,\bbp)$ be a relatively hyperbolic group with tame peripherals. Assume that $\relbndryG$ is connected and not homeomorphic to a circle. If $G$ does not split over a $2$-ended subgroup, then $\relbndryG$ does not contain a non-parabolic local cut point. Moreover, if $G$ splits over a non-parabolic $2$-ended subgroup relative to $\bbp$, then $\relbndryG$ contains a non-parabolic local cut point.
\end{theorem}

A relatively hyperbolic group $(G,\bbp)$ has {\it tame peripherals}, if  every $P\in\bbp$ is finitely presented, one- or two-ended, and contains no infinite torsion subgroup. Bowditch has shown \cite{Bow01} that if $(G,\bbp)$ has tame peripherals and the Bowditch boundary $\relbndryG$ is connected, then $\relbndryG$ is locally connected. In this paper we will always assume that $\relbndryG$ is connected and that $(G,\bbp)$ has tame peripherals. Other terms used in the statement of Theorem \ref{theorem:Splitting Theorem} will be defined in Section \ref{section:Preliminaries}.

Kapovich and Kleiner's classification result for $1$-dimensional boundaries of hyperbolic groups \cite{KK00} shows that under certain group theoretic conditions if the Gromov boundary is $1$-dimensional, then it must be a circle, a Sierpinski carpet, or a Menger curve. Theorem \ref{theorem:Splitting Theorem} is used by the author in \cite{Hau17b} to generalize the Kapovich-Kleiner result to $1$-dimensional visual boundaries of CAT(0) groups with isolated flats which do not split over $2$-ended subgroups. (We point out that visual boundary and the Bowditch boundary are not the same in general \cite{Tran}.) The application of Theorem \ref{theorem:Splitting Theorem} is a critical step in the proof in Theorem 1.2 of \cite{Hau17b} and requires an understanding of the general case where the Bowditch boundary has global cut points.

In the present paper we use Theorem \ref{theorem:Splitting Theorem} to obtain an alternate generalization of the Kapovich-Kleiner Theorem for $1$-dimensional Bowditch boundaries of relatively hyperbolic groups with $1$-ended and tame peripherals.

\begin{theorem}[Classification Theorem]
\label{theorem: Classification thm}
Let $(G,\bbp)$ be a $1$-ended relatively hyperbolic group with tame peripherals, and let $\mathcal{P}$ be the set of all subgroups of elements of $\bbp$. Assume that $G$ does not split over a virtually cyclic subgroup and does not split over any subgroup in $\mathcal{P}$. If every $P\in\bbp$ is one-ended and $\relbndryG$ is 1-dimensional, then one of the following holds:
\begin{enumerate}
\item $\relbndryG$ is a circle
\item $\relbndryG$ is a Sierpinski carpet
\item $\relbndryG$ is a Menger curve.
\end{enumerate}
\end{theorem}

A $\CAT(0)$ group with isolated flats is relatively hyperbolic with respect virtually abelian subgroups \cite{HK1} and thus is relatively hyperbolic with respect to a collection of tame $1$-ended peripherals. However, there are distinctions between Theorem \ref{theorem: Classification thm} of this paper and Theorem 1.2 of \cite{Hau17b} worth mentioning here. First and foremost, the Bowditch boundary $\relbndryG$ is generally a quotient space of the visual boundary \cite{Tran}. Second, peripheral splittings (see Section \ref{subsec:splittings}) play a role in both theorems.  Peripheral splittings are not allowed for Theorem \ref{theorem: Classification thm}, whereas in the isolated flats setting only certain types of peripheral splittings are excluded. Lastly, in the $\CAT(0)$ setting the boundary has no global cut points \cite{PS09}, but the Bowditch boundary of a relatively hyperbolic group may have many global cut points in general.

\subsection{Method of proof}
\label{subsec:method of proof}

The proof of Theorem \ref{theorem:Splitting Theorem} utilizes arguments of Bowditch \cite{Bow98} developed for hyperbolic groups; however, because we are interested in the relatively hyperbolic setting and Bowditch's results depend on hyperbolicity in an essential way additional techniques are required. In particular, in Lemmas 5.2 and 5.17 of \cite{Bow98} are key steps in which Bowditch explicitly uses hyperbolicity. Using results of Tukia \cite{Tukia98}, Guralnik \cite{Gur} proved a relatively hyperbolic version of Bowditch's Lemma 5.2 which may be found as Lemma \ref{lemma:control theorem} in this exposition. In Section \ref{subsec:first key lemma} we provide a simple self-contained proof of Lemma \ref{lemma:control theorem} using techniques different from those of \cite{Gur}. Guralnik also observed that given Lemma \ref{lemma:control theorem} and given the Bowditch boundary has no global cut points, then some of Bowditch's local cut points carry over to the relatively hyperbolic setting using Bowditch's exact arguments. In Lemma 5.17 \cite{Bow98} Bowditch shows that the stabilizer of a necklace (see Section \ref{subsec:Cut Point Structures In Metric Spaces} for the definition of a necklace) in the boundary of a relatively hyperbolic group is quasiconvex. Proposition \ref{prop:Key prop} provides a relatively hyperbolic version of this result. Namely, we show that the stabilizer of a necklace in the Bowditch boundary of a relatively hyperbolic group is relatively quasiconvex. The importance of Lemma \ref{lemma:control theorem} and Proposition \ref{prop:Key prop} is that they allow us to use Bowditch's arguments verbatim to determine the local cut point structure of the Bowditch boundary in the special case that the Bowditch boundary does not contain any global cut points.

If the boundary of a relatively hyperbolic group $(G,\bbp)$ is connected, then it is a Peano continuum. However, $\relbndryG$ may have many global cut points \cite{Bow01}. Our strategy involves demonstrating that it suffices to consider only the case when $\relbndryG$ has no global cut points. In particular, using the theory of peripheral splittings \cite{Bow01} and basic decomposition theory we are able to restrict our attention to ``blocks'' of $\relbndryG$, where a {\it block} of $\relbndryG$ is a maximal subcontinuum consisting of points which cannot be separated from one another by global cut points. Blocks have two key features. The first is that a block of $\relbndryG$ is the limit set of a relatively hyperbolic subgroup $(H,\bbq)$ of $(G,\bbp)$ (see Theorem \ref{theorem:treelike structure}(4)). The second is that there is a retraction of  $\relbndryG$ onto any given block; moreover, the retraction map has nice decomposition theoretic properties. This combination of Bowditch's theory of peripheral splittings with decomposition theory techniques is one of the major contributions of this paper, and it is the focus of Section \ref{section:Reduction}. Using these techniques allows us to reduce the proof of Theorem \ref{theorem:Splitting Theorem} to proving Theorem \ref{theorem:collecting conical}, which describes non-parabolic local cut points in a boundary without global cut points. The proof of Theorem \ref{theorem:collecting conical} can be found in Section \ref{subsec:collecting local} and relies on the observation from the above paragraph that Lemma \ref{lemma:control theorem} and Proposition \ref{prop:Key prop} allow us to directly use the arguments \cite{Bow98}.

The other main result of this paper is Theorem \ref{theorem: Classification thm}. Two key tools used in the proof of Theorem \ref{theorem: Classification thm} are the topological characterization of the Menger curve due to R.D. Anderson \cite{A58a, A58b} and the topological characterization of the Sierpinski carpet due to Whyburn \cite{Whyburn}. Anderson's theorem states that a compact metric space $M$ is a Menger curve provided $M$ is $1$-dimensional, $M$ is connected, $M$ is locally connected, $M$ has no local cut points, and no non-empty open subset of $M$ is planar. We note that if the last condition is replaced with ``$M$ is planar,'' then we have the topological characterization of the Sierpinski carpet (see Whyburn \cite{Whyburn}).

In order to apply Anderson and Whyburn's theorems we must rule out the existence of local cut points. Theorem \ref{theorem:Splitting Theorem} can be used to rule out non-parabolic local cut points, but we also need to rule out the existence of parabolic local cut points. In Theorem \ref{theorem: Classification thm} we are in a setting where $\relbndryG$ contains no global cut points, so $\relbndryG\setminus\{p\}$ is connected. Thus we need only know that $\relbndryG\setminus\{p\}$ is $1$-ended. Because the group $P=\stab(p)$ is $1$-ended, and Bowditch \cite{Bow1} has shown that $P$ acts properly and cocompactly on $\relbndryG\setminus\{p\}$, a reader familiar with geometric group theory may think that we are done. However, the author was unable to find sufficiently general results in the literature. To define ends of a group one must make a choice of a space on which the group acts, and it is well known that any two CW-complexes on which $G$ acts properly and cocompactly have the same number of ends \cite{Geo, Gui13}.

 In this paper we require an understanding of the ends of a connected open subset of a Peano continuum on which a group acts properly and cocompactly. The study of ends as introduced by Freudenthal \cite{Freu31} can be described as inverse limits in the setting of generalized continua (i.e. locally compact, locally connected, $\sigma$-compact, connected Hausdorff spaces) as explained by Baues and Quintero \cite{BQ01}. Given a finitely generated discrete group $G$ acting properly and cocompactly on two ``nice'' topological spaces $X$ and $Y$ a natural question to ask is: What topological properties are required to guarantee that $\Ends(X)$ homeomorphic to $\Ends(Y)$? In other words, what is the natural setting in which $\Ends(G)$ is well defined?

 A {\it generalized Peano continuum} is a generalized continuum which is metrizable. This general class of spaces includes open connected subspaces of Peano continua, proper geodesic metric spaces, and locally finite CW-complexes. Theorem \ref{theorem:ends theorem} extends known ends results to generalized Peano continua. For groups acting properly and compactly on metric spaces with a proper equivariant geodesic metric, Theorem \ref{theorem:ends theorem} follows from Proposition I.8.29 of \cite{BH1}. However, it is unknown whether a generalized Peano continuum with a proper, cocompact action admits an equivariant proper geodesic metric (see Section 6 of \cite{GuiMor17}). Thus Theorem \ref{theorem:ends theorem} is a new contribution the literature. Theorem \ref{theorem:ends theorem} has already proved useful outside of this paper. Groves and Manning (see Section 7 of \cite{GM16}) use Theorem \ref{theorem:ends theorem} to prove a result similar to Theorem \ref{theorem:Splitting Theorem} for the restricted case where $\relbndryG$ has no global cut points, and the peripheral subgroups are $1$-ended and tame.

 \begin{theorem}
\label{theorem:ends theorem}
Let $X$ and $Y$ be two generalized Peano continua, and assume that $G$ is a finitely generated group acting properly and cocompactly by homeomorphisms on $X$ and $Y$. Then $\Ends(X)$ is homeomorphic to $\Ends(Y)$.
\end{theorem}

A generalized Peano continuum $X$ is locally path connected (see for example Exercise 31C.1 of \cite{Wil1}). The proof of Theorem \ref{theorem:ends theorem} depends heavily on this fact. In particular, the existence of proper rays in $X$ plays an important role in the proof of Theorem \ref{theorem:ends theorem}. The author would be interested to know if there is an alternate argument that could be extended from the metric setting to the more general, non-metric, setting of generalized continua.

\begin{problem}
Let $X$ and $Y$ be two generalized continua, and $G$ be a finitely generated discrete group acting properly and cocompactly on $X$ and $Y$. Is $\Ends(X)$ homeomorphic to $\Ends(Y)$?
\end{problem}

Theorem \ref{theorem:ends theorem} is also closely related to the recent work of Guilbault-Moran \cite{GuiMor17}. Their work implies Theorem \ref{theorem:ends theorem} for geometric actions on proper metric ARs.



\subsection{Acknowledgments}
First I would like to thank my advisor Chris Hruska for providing comments and guidance throughout this project. Second I would like to give special thanks to Genevieve Walsh for encouraging me to write up the details of these results. Lastly, I would like to thank Craig Guilbault for valuable conversations regarding Freudenthal ends.

\section{Preliminaries}
\label{section:Preliminaries}
In this section we review basic facts and definitions required by the exposition of this paper. The topics in this section include convergence groups, relatively hyperbolic groups, splittings, cut point structures in metric spaces, and ends of generalized continua.

\subsection{Convergence group actions}
\label{subsection:Convergence Groups}
A detailed account of convergence group actions may be found in \cite{Bow99d}. Let $M$ be a compact metrizable space. Let $G$ be a group acting by homeomorphisms on $M$. The group $G$ is called a {\it convergence group} if for every sequence of distinct group elements $(g_k)$ there exist points $\alpha, \beta\in M$ (not necessarily distinct) and a subsequence $(g_{k_i})\subset (g_k)$ such that $g_{k_i}(x)\rightarrow\alpha$ locally uniformly on $M\setminus\{\beta\}$, and $g_{k_i}^{-1}(x)\rightarrow\beta$ converges locally uniformly on $M\setminus\{\alpha\}$. By {\it locally uniformly} we mean, if $C$ is a compact subset of $M\setminus\{\beta\}$ and $U$ is any open neighborhood of $\alpha$, then there is an $N\in\bbn$ such that $g_{k_i} C\subset U$ for all $i>N$.

Elements of convergence groups can be classified into three types: elliptic, loxodromic, and parabolic. A group element is {\it elliptic} if it has finite order. An element $g$ of $G$ is {\it loxodromic} if has infinite order and fixes exactly two points of $M$. A subgroup of $G$ is loxodromic if it is virtually infinite cyclic. If $g\in G$ has infinite order and fixes a single point of $M$ then $g$ is {\it parabolic}. An infinite subgroup $P$ of $G$ is {\it parabolic} if it contains no loxodromic elements and stabilizes a single point $p$ of $M$.  The point $p$ is uniquely determined by $P$, and the point $p$ is called a {\it parabolic} point. We call $p$ a {\it bounded parabolic} point if $P$ acts properly and cocompactly on $M\setminus\{p\}$.

A point $x\in\relbndryG$ is a {\it conical} {\it limit} point if there exists a sequence of group elements $(g_n)\in G$ and distinct points $\alpha,\beta\in M$ such that $g_nx\rightarrow \alpha$ and $g_ny\rightarrow\beta$ for every $y\in M\setminus\{x\}$. Tukia has shown (see \cite{Tukia98}) that:

\begin{proposition}
\label{prop:conical vs parabolic}
 A conical limit point cannot be a parabolic point
\end{proposition}

\medskip
 A convergence group $G$ acting on $M$ is called {\it uniform} if every point of $M$ is a conical limit point, or equivalently the action on space of distinct triples of $M$ is proper and cocompact (see \cite{Bow99d}). Bowditch has shown \cite{Bow98b} $G$ is a uniform convergence group if and only if it is hyperbolic. $G$ is called {\it geometrically finite} if every point of $M$ is a conical limit point or a bounded parabolic point.

\subsection{Relatively hyperbolic groups and their boundaries}
\label{subsec:relhypgrps}

We refer the reader to \cite{Bow1} for a more thorough introduction to relatively hyperbolic groups. Let $G$ be a group acting properly and isometrically on a $\delta$-hyperbolic space $X$. Tukia \cite{Tuk94} has shown that $G$ acts on the Gromov boundary of $X$ as a convergence group. Let $\mathbb{P}$ a collection of infinite subgroups of $G$ that is closed under conjugation, called {\it peripheral subgroups}.
\medskip

\Def We say that $G$ is {\it hyperbolic relative to} $\bbp$ if:
\begin{enumerate}
\item $\bbp$ is the set of all maximal parabolic subgroups of $G$
\item There exists a $G$-invariant system of disjoint open horoballs based at the parabolic points of $G$, such that if $\mathcal{B}$ is the union of these horoballs, then $G$ acts cocompactly on $X\setminus\mathcal{B}$.
\end{enumerate}

Any action of a group $G$ on a proper $\delta$-hyperbolic space satisfying the above definition is called {\it cusp uniform}. In \cite{Bow1} Bowditch shows:

\begin{theorem}
\label{theorem:finitely many orbits of maximal parabolic}
 If $G$ is hyperbolic relative to $\bbp$, then $\bbp$ consists of only finitely many conjugacy classes.
\end{theorem}

The Bowditch boundary $\partial (G,\bbp)$ is defined to be the Gromov boundary of $X$, i.e the set of equivalence classes of geodesic rays of $X$, where two geodesic rays are equivalent if their Hausdorff distance is finite. It is a result of Bowditch \cite{Bow1} that $\relbndryG$ does not depend on the choice of $X$.

We say that a relatively hyperbolic group $(G,\bbp)$ has {\it tame peripherals} if  every $P\in\bbp$ is finitely presented, one- or two-ended, and contains no infinite torsion subgroup. Under the assumption of tame peripherals Bowditch has shown the following two results in \cite{Bow99a} and \cite{Bow01}, respectively.

\begin{theorem}
Suppose that $(G,\bbp)$ is relatively hyperbolic with tame peripherals and that $\relbndryG$ is connected. Then every global cut point of $\relbndryG$ is a parabolic point.
 \end{theorem}

Assume that $\relbndryG$ is connected. A {\it global cut point} is a point whose removal disconnects $\relbndryG$.

\begin{theorem}
\label{theorem:Bow Local conn}
 If $(G,\bbp)$ is relatively hyperbolic with tame peripherals and $\relbndryG$ is connected, then $\relbndryG$ is locally connected.
\end{theorem}

In this paper we are interested in the case where $\relbndryG$ is locally connected, so we will generally assume that $(G,\bbp)$ has tame peripherals and that $\relbndryG$ is connected.

\begin{remark}[Convergence Groups and Relatively Hyperbolic Groups]
Recall from Section \ref{subsection:Convergence Groups} that a convergence group action of a group $G$ is uniform if and only if $G$ is hyperbolic. A generalization of this result was completed by Bowditch \cite{Bow1} and Yaman \cite{Yam04}. Bowditch \cite{Bow1} shows that a relatively hyperbolic group with finitely generated peripheral subgroups acts on its Bowditch boundary as a geometrically finite convergence group (See Proposition \ref{proposition:cusp uniform is geom fin}), and Yaman \cite{Yam04} proves a strong converse. We remark that in general geometrically finite convergence group actions are not uniform (see Proposition \ref{prop:conical vs parabolic}).
\end{remark}

\begin{proposition}[Bowditch \cite{Bow1} Proposition 6.12]
\label{proposition:cusp uniform is geom fin}
Assume $G$ acts properly and isometrically on a proper $\delta$-hyperbolic space $X$, and let $\bbp$ be a collection of infinite subgroups of $G$. If the action of $(G,\bbp)$ on $X$ is cusp uniform, then the action on $\bndry X$ is geometrically finite.
\end{proposition}

\subsection{Splittings}
\label{subsec:splittings}

A graph of groups is called {\it trivial} \cite{Serre, SW79} if there exists a vertex group equal to $G$. A {\it splitting} of a group $G$ over a given class of subgroups is a non-trivial finite graph of groups representation of $G$, where each edge group belongs to the given class. We say that $G$ {\it splits over} a subgroup $A$ if $G$ splits over the class $\{A\}$. The group $G$ is said to split {\it relative} to another class of subgroups $\bbp$ if each element of $\bbp$ is conjugate into one of the vertex groups. Assume that $G$ is hyperbolic relative to a collection $\bbp$. A {\it peripheral splitting} of $(G, \bbp)$ is a finite bipartite graph of groups representation of $G$, where $\bbp$ is the set of conjugacy classes of vertex groups of one color of the partition called peripheral vertices. Non-peripheral vertex groups will be referred to as {\it components}. This terminology stems from tame peripheral setting where there is a correspondence between the cut point tree of $\relbndryG$ and the peripheral splitting of $(G, \bbp)$. In this correspondence elements of $\bbp$ correspond to stabilizers of cut point vertices and the components correspond to stabilizers of blocks in the boundary (see Theorem \ref{theorem:treelike structure}).

A peripheral splitting $\mathcal{G}$ is a refinement of another peripheral splitting $\mathcal{G}'$ if $\mathcal{G}'$ can be obtained from $\mathcal{G}$ via a finite sequence of foldings that preserve the vertex coloring. In \cite{Bow01} Bowditch proves the following accessibility result:

\begin{theorem}
\label{theorem:Bowditch Accessibility}
Suppose that $(G,\bbp)$ is relatively hyperbolic with tame peripherals and connected boundary.
Then $(G,\bbp)$ admits a (possibly trivial) peripheral splitting which is maximal in the sense
that it is not a refinement of any other peripheral splitting.
\end{theorem}

Combining Proposition 5.1 and Theorem 1.2 of \cite{Bow99b} Bowditch also shows:

\begin{theorem}
\label{theorem:global cut point implies splitting}
 If $(G,\bbp)$ is a relatively hyperbolic with tame peripherals, $\relbndryG$ is connected, and $\relbndryG$ has a global cut point, then there exists a non-trivial peripheral splitting of $(G,\bbp)$.
\end{theorem}

\subsection{Local cut point structures in Peano continua}
\label{subsec:Cut Point Structures In Metric Spaces}

We refer the reader to \cite{Bow98} for a more detailed account of local cut point structures in Peano continua. Recall that a {\it Peano continuum} is a compact, connected, and locally connected metric space. Let $M$ be a Peano continuum. A {\it global cut point} of $M$ is a point $x\in M$ such that $M\setminus\{x\}$ is disconnected. A {\it cut pair} is a set of two distinct points $\{a,b\}\subset M$ which contains no global cut points, and such that $M\setminus\{a,b\}$ is disconnected. The set of components of $M\setminus\{a,b\}$ will be denoted by $\U(a,b)$, and $\N(a,b)$ will denote the cardinality of $\U(a,b)$. We leave it as an exercise to show if $x$ is a global cut point and $\{a,b\}$ is a cut pair, then $a$ and $b$ cannot lie in different components of $M\setminus\{x\}$. Two cut pairs $\{a,b\}$ and $\{c,d\}$ are said to {\it mutually separate} $M$ if $c$ and $d$ lie in different components of $M\setminus\{a,b\}$ and vice versa. A cut pair $\{a,b\}$ is called {\it inseparable} if there does not exist any other cut pair $\{c,d\}$ such that $a$ and $b$ lie in distinct components of $M\setminus\{c,d\}$. Let $G$ be a group acting on $M$ by homeomorphisms. A cut pair $\{a,b\}$ will be called {\it translate inseparable} (or {\it $G$-translate inseparable}) if there does not exist a cut pair $\{c,d\}$ in $\Orb_G\big(\{a,b\}\big)$ such that $a$ and $b$ lie in distinct components of $M\setminus\{c,d\}$. If $(G,\bbp)$ is relatively hyperbolic and $M=\relbndryG$, then a cut pair $\{a,b\}$ will be called {\it loxodromic} if it is stabilized by a loxodromic element $g\in G$.

Let $\Delta$ be a subset of $M$. A {\it cyclic order} on $\Delta$ is a quaternary relation $\sigma(a,b,c,d)$, such that the following holds. If $F$ is any finite subset of $\Delta$, then there is an embedding $i\colon F\rightarrow S^1$ so that given any four $a,b,c,d\in F$ the relation $\sigma(a,b,c,d)$ holds if and only if the pairs $\bigl\{i(a),i(c)\bigr\}$ and $\{i(b),i(d)\}$ mutually separate in $S^1$. The set $\Delta$ is called a {\it cyclically separating set} if it has a cyclic order and $i$ induces a map from the components of $M\setminus\{a,b,c,d\}$ onto the components of $S^1\setminus\bigl\{i(a),i(b),i(c),i(d)\bigr\}$, where a component $C$ of $M\setminus\{a,b,c,d\}$ with frontier $Fr(C)=\{x,y\}\subset\{a,b,c,d\}$ is mapped to the component of $S^1\setminus\big\{i(a),i(b),i(c),i(d)\big\}$ with frontier $\bigl\{i(x),i(y)\bigr\}$. Two points $a$ and $b$ in $\Delta$ are called {\it adjacent} if $\big\{i(a),i(b)\big\}$ cannot be mutually separated by $\big\{i(c),i(d)\big\}$ for any $c,d\in\Delta$. An unordered pair of adjacent points will be referred to as a {\it jump}. Notice that if two distinct jumps of $\Delta$ intersect, they must intersect in an isolated point of $\Delta$.

\medskip

A point $x\in M$ is a {\it local cut point} if $M\setminus\{x\}$ is disconnected or $M\setminus\{x\}$ is connected and has more than one end. If $M\setminus\{x\}$ is connected the {\it valence}, $val(x)$, of a local cut point is the number of ends of $M\setminus\{x\}$. A detailed discussion of ends of spaces can be found in Section \ref{subsection: Ends of Spaces}, but we remark that saying a point $x\in M$ is a local cut point is equivalent to saying that there exists a neighborhood $U$ of $x$ such that for every neighborhood $V$ of $x$ with $V\subset U$, there exist points $z,y\in V\setminus\{x\}$ such that there does not exist a connected subset of $U\setminus\{x\}$ containing $z$ and $y$. Alternatively, to check that $x$ is not a local cut point it suffices to show that given a neighborhood $U$ of $x$ there exists a neighborhood $V\ni x$ with $V\subset U$ and $V\setminus \{x\}$ connected. We wish to ``collect'' all the local cut points, so we introduce notation similar to that of Bowditch \cite{Bow98} to describe the various ``local cut point structures'' in $M$.  Let $M(n)=\bigset{x\in M}{\spc val(x)=n, \text{and}\, x\, \text{is not a global cut point}}$ and $M(n+)=\bigset{x\in M}{\spc val(x)\geq n, \text{and}\, x\, \text{is not a global cut point}}$.

 Now assume that a group $G$ acts on $M$ with a geometrically finite convergence group action. Then $G$ is relatively hyperbolic and $M$ is homeomorphic to $\relbndryG$ \cite{Bow98b} \cite{Yam04}. If $M=\relbndryG$, then $M$ consists entirely of conical limit points and parabolic points. Moreover, if $(G,\bbp)$ has tame peripherals, then global cut points in $M$ correspond to parabolic points (see Section \ref{subsection:Convergence Groups}). Because parabolic points cannot be conical limit points (Proposition \ref{prop:conical vs parabolic}), the goal is to understand local cut points which are conical limit points to ensure that the points we are considering do not separate $M$ globally. Define $\C$ to be the collection of conical limit points in $M$. We will denote by $M^*(n)$ and $M^*(n+)$ the intersections of $M(n)$ and $M(n+)$ with $\C$. We define relations on $M^*(2)$ and $M^*(3+)$. Let $x,y\in M^*(2)$. We write $x\sim y$ if and only if $x=y$ or $\N(x,y)=2$. For two elements $a,b\in M^*(3+)$ we write $a\approx b$ if $a\neq b$ and $\N(a,b)=val(a)=val(b)\geq 3$. From the definitions above we immediately obtain a partition of the set of conical limit points which are local cut points. In other words:

\begin{lemma}
\label{lemma:partition conical points}
Let $x\in M$ be a conical limit point which is a local cut point. Then $x\in M^*(2)\cup M^*(3+)$.
\end{lemma}

\medskip
The following results are proved  using the same arguments as those of Bowditch in \cite{Bow98}:

\begin{lemma}
\label{lemma:*3.8 doubleTilde lemma}
 The collection of $\approx$-classes in $M^*(3+)$ is partitioned into pairs, which do not mutually separate.
\end{lemma}

\begin{lemma}
\label{lemma:*3.1}
The relation $\sim$ is an equivalence relation on $M^*(2)$.
\end{lemma}

We say that a cut pair $\{c,d\}$ in $M$ {\it separates} a subset $C\subset M$ if $C$ is contained in at least two distinct components of $M\setminus\{c,d\}$.

\begin{lemma}
\label{lemma: separate implies cyclic}
Let $a,b,c,d\in M^*(2)$. If $a\sim b$ and $\{c,d\}$ separates $\{a,b\}$, then $c\sim d\sim a\sim b$, and the pairs $\{a,b\}$ and $\{c,d\}$ mutually separate.

\end{lemma}

\medskip

An argument similar to that of Bowditch \cite{Bow98} shows that there are no singleton $\sim$-classes in $M^*(2)$; consequently, a $\sim$-class in $M^*(2)$ consists of either a cut pair or a cyclically separating collection of cut pairs. The closure of a $\sim$-class $\nu$ containing at least three elements will be called a {\it necklace}. Notice that if $\nu$ is infinite, then $\overline{\nu}$ may contain parabolic points. Lastly remark that, because cut pairs cannot be separated by global cut points neither can $\sim$-classes or their closures.

\subsection{Ends of generalized continua}
\label{subsection: Ends of Spaces}

In this section we review ends of spaces. Roughly speaking the number of ends of a connected space $X$ counts the number of components at infinity in $X$. A more detailed discussion about ends of spaces may be found in Section 3 of \cite{Gui13} and  Section I.9 of \cite{BQ01}.

A {\it continuum} is a compact, connected, locally connected Hausdorff space. A {\it generalized continuum} is a connected, locally compact, locally connected, $\sigma$-compact, Hausdorff space. A {\it generalized Peano continuum} is a metrizable generalized continuum. A nested sequence $C_1\subseteq C_2 \subseteq C_3\subseteq\cdots$ of compact sets in a space $X$ is called an {\it exhaustion of X} if $X=\mathop{\bigcup}_{i=1}^{\infty} C_i $ and $C_i\subset\Int(C_{i+1})$ for every $i$. Note that $\sigma$-compactness implies that a generalized continuum can always be covered by a sequence of compact sets, and by local compactness we may always assume that $C_i\subset\Int(C_{i+1})$.  Also note that, for generalized Peano continua the components of the complement of a compact set are path components of the complement. The context of this paper makes it worth noting that a connected open subset of a Peano continuum is a generalized Peano continuum. 

Let $X$ be a generalized continuum and let $\{C_i\}_{i=1}^{\infty}$ be an exhaustion of $X$. Define $\mathcal{U}(C_i)$ to be the set of components of $X\setminus C_i $. Because the sequence $\{C_i\}_{i=1}^{\infty}$ is nested the sets $\U(C_i)$ form an inverse sequence:\[ \U(C_1)\leftarrow\U(C_2)\leftarrow\U(C_3)\leftarrow\cdots \]
The set $\Ends(X)$ is defined to be $\varprojlim\bigset{\U(C_i)}{i\geq 1}$. The cardinality of $\Ends(X)$ is the {\it number of ends} of the space $X$. The set $\Ends(X)$ is independent of choice of $\{C_i\}$ (see Remark I.9.2 (a) of \cite{BQ01}). Let $G$ be a finitely generated discrete group acting properly and cocompactly on a generalized Peano continuum $X$. We define $\Ends(G)$ to be $\Ends(X)$. We show in Theorem \ref{theorem:ends theorem} that $\Ends(G)$ is independent of the choice of $X$ and agrees with the more traditional notion of ends of a Cayley graph for any finite generating set.

\medskip
The {\it Freudenthal compactification} of $X$ is $X\cup \Ends(X)$ with the topology generated by the basis consisting of all open subsets of $X$ and all sets $\overline{E}_i$ where $E_i\in\U(C_i)$ for some $i\geq 1$ and
\[\overline{E}_i=E_i\cup\bigset{(F_1,F_2,F_3,...)\in \Ends(X)}{\spc F_i=E_i}.\]
It is well known that the Freudenthal compactification is compact and metrizable. The space $\Ends(X)$ is given the subspace topology.

\medskip
Recall that a map between two spaces $f\colon X\rightarrow Y$ is called {\it proper} if for every compact subset $C$ of $Y$ we have $f^{-1}(C)$ is compact. The following well known result can be found as Proposition I.9.11 of \cite{BQ01}.

\begin{proposition}
\label{proposition:Craig's prop}
 Let $f\colon X\rightarrow Y$ be a proper map between generalized Peano continua, then $f$ can be uniquely extended to a continuous map $\hat{f}$ from $X\cup\Ends(X)$ to $Y\cup\Ends(Y)$.
\end{proposition}

The restriction of $\hat{f}$ to $\Ends(X)$ will be denoted by $f^*$, and we say that $f^*$ is the {\it the ends map induced} by $f$.

\medskip
A useful and more geometric way to describe the ends of a generalized Peano continuum $X$ is by proper rays. By {\it proper ray} we mean any proper map $\alpha\colon [0,\infty)\rightarrow X$. Two rays $\alpha$ and $\beta$ are {\it ladder equivalent} if there is a proper map $h$ of the {\it infinite ladder} (or simply {\it ladder})
\[L_{[0,\infty)}=\big([0,\infty)\times \{0,1\}\big)\cup \big(\bbn\times [0,1]\big)\]
such that $\alpha$ and $\beta$ are the {\it sides}, i.e. $\alpha=h|_{[0,\infty)\times \{0\}}$ and $\beta=h|_{[0,\infty)\times \{1\}}$. We will write $\alpha\simeq\beta$ to denote that $\alpha$ is ladder equivalent to $\beta$. The image under $h$ of $n\times [0,1]$ is called a {\it rung}. Let $\mathcal{L}(X)$ be the collection of ladder classes of proper rays in $X$.

Assume $\alpha$ is a proper ray in $X$. By Proposition \ref{proposition:Craig's prop} there is a continuous extension $\hat{\alpha}\colon [0,\infty)\cup\{\infty\}\rightarrow X\cup\Ends(X)$ of $\alpha$, so that $\alpha(\infty)$ is an element of $\Ends(X)$. By Proposition I.9.20 of \cite{BQ01} we have:

\begin{proposition}
\label{prop:ladder prop}
Let be $X$ a generalized Peano continuum. The map $\varphi\colon\mathcal{L}(X)\rightarrow \Ends(X)$ given by setting $\varphi\big([\alpha]_{\simeq}\big)$ equal to $E=\hat{\alpha}(\infty)$ defines a $1-1$ correspondence between $\Ends(X)$ and $\mathcal{L}(X)$.
\end{proposition}

\begin{lemma}
\label{lemma: ladder equivalence}
Let $\hat{f}\colon X\cup\Ends(X)\rightarrow Y\cup\Ends(Y)$ be the continuous extension of a proper map $f$ between two generalized Peano continua $X$ and $Y$, and let $f^*$ denote the restriction of $f$ to $\Ends(X)$. Assume that $\varphi_1\colon\mathcal{L}(X)\rightarrow\Ends(X)$ and $\varphi_2\colon\mathcal{L}(Y)\rightarrow\Ends(Y)$ are the bijections as given in Proposition \ref{prop:ladder prop}. Let the map $g\colon\mathcal{L}(X)\rightarrow\mathcal{L}(Y)$ be given by  $g([\alpha])=[f\alpha]$ for every $[\alpha]\in\mathcal{L}(X)$. Then the following diagram commutes:

\[
  \begin{tikzcd}
    \mathcal{L}(X) \arrow{r}{g} \arrow[swap]{d}{\varphi_1} & \mathcal{L}(Y) \arrow{d}{\varphi_2} \\
   \Ends(X)\arrow{r}{f^*} & \Ends(Y)
  \end{tikzcd}
\]
\end{lemma}

\begin{proof}
Assume the hypotheses and let $\alpha\colon[0,\infty)\rightarrow X$ be a proper ray.  By Proposition \ref{proposition:Craig's prop} there is a continuous extension $\hat{\alpha}\colon [0,\infty]\rightarrow X\cup\Ends(X)$ of $\alpha$, and there is a continuous extension $\widehat{f\alpha}\colon [0,\infty]\rightarrow Y\cup\Ends(Y)$ of $f\alpha$. The composition $\hat{f}\hat{\alpha}$ is also a continuous extension of $\alpha$ to $[0,\infty]$ with range space $Y$. The subspace $[0,\infty)$ is dense in $[0,\infty]$, and extensions of continuous maps into Hausdorff spaces from dense subspaces to their closures are unique. So, we have $\hat{f}\hat{\alpha}(\infty)$ must equal $\widehat{f\alpha}(\infty)$. Thus $\varphi_2g\big([\alpha]\big)=\varphi_2(f\alpha)=\widehat{f\alpha}(\infty)=\hat{f}\hat{\alpha}(\infty)=f^*\big(\hat{\alpha}(\infty)\big)=f^*\varphi_1\big([\alpha]\big)$.
\end{proof}

\subsection{Limit sets, joins and relative quasiconvexity}
In this section let $X$ be a $\delta$-hyperbolic space, let $(G,\bbp)$ be a group acting on $X$ is a cusp uniform action, and let $H$ be any subgroup of $(G,\bbp)$. For a sequence $(h_n)\subset H$ we write $h_n\rightarrow \xi\in\bndry X$ if $h_n x\rightarrow\xi$ for some $x\in X$. Note that if $h_n x\rightarrow\xi$ for some $x$, then $h_n x'\rightarrow\xi$ for any $x'\in X$. The {\it limit set} $\Lambda(H)$ of $H$ is the subset of $\bndry X$ consisting of all limits such limits. The set $\Lambda(H)$ is a closed and $H$-invariant.

Given a subset $\Omega$ of $\bndry X$ containing at least two points, we will denote by $\Jn(\Omega)$ the union of all geodesic lines joining points of $\Omega$. If $\Omega$ is closed, then it follows from a standard diagonal argument that $\Jn(\Omega)$ is closed. The space $\Jn(\Omega)$ is quasi-isometric to a geodesic Gromov hyperbolic space $\gJn(\Omega)$ 
(See \cite{Gro87} Section 7.5.A). In fact, $\gJn(\Omega)$ is finite neighborhood of $\gJn(\Omega)$ in the space $X$ endowed with the length metric. An infinite subgroup $H$ of $(G,\bbp)$ is {\it relatively quasiconvex} if $H$ is parabolic, or $(H,\bbq)$ has a cusp uniform action on $\gJn\big(\Lambda(H)\big)$ where $\bbq=\bigset{Q}{Q=H\cap P\spc\ {\it with}\spc  Q\spc {\it infinite} \spc {\it and}\spc P\in\bbp}$.

The following may be found as Proposition 7.1 of \cite{Hru10}.

\begin{proposition}
\label{proposition:Hruska qc1vqc2}
A subgroup $H$ of $(G,\bbp)$ is relatively quasiconvex if and only if the induced convergence action of $H$ on $\Lambda(H)$ is geometrically finite.
\end{proposition}

We will implicitly be using the following proposition through out Section \ref{section:Reduction} and Section \ref{section:Local Cut Points in relbndry}. Proposition \ref{proposition:Hruska Thm 9.1} is distillate from proof of Theorem 9.1 of \cite{Hru10}.

\begin{proposition}
\label{proposition:Hruska Thm 9.1}
Let $H$ be a relatively quasiconvex subgroup of $G$, and let $x\in\Lambda(H)$ then the following hold:
\begin{enumerate}
\item $x$ is a conical limit point of the induced action of $H$ on $\Lambda(H)$ if and only if it is a conical limit point of the action of $G$ on $\relbndryG$.
\item $x$ is a bounded parabolic point of the induced action of $H$ on $\Lambda(H)$ if and only if it is a bounded parabolic point of the action of $G$ on $\relbndryG$.
\end{enumerate}
\end{proposition}

\medskip
\section{Reduction}
\label{section:Reduction}
Let $(G,\bbp)$ be a relatively hyperbolic group with tame peripherals. The results in this section can be considered the first step in the proof of Theorem \ref{theorem:Splitting Theorem}. In particular, we show that the proof of Theorem \ref{theorem:Splitting Theorem} can be reduced to the case where the Bowditch boundary $\relbndryG$ has no global cut points.
\subsection{Blocks and branches}
\label{subsection: Blocks and Branches}
In this subsection we look at cut point decompositions of $\relbndryG$. For a more in depth overview see \cite{Bow99c} and \cite{Swe99}.

Let $M$ be a Peano continuum, and let $\Pi$ be the set of global cut points of $M$. We define a relation $R$ on $M$ by $xRy$ if $x$ and $y$ cannot be separated by an element of $\Pi$. In other words, $xRy$ means there does not exist a $z\in\Pi$ such that $x$ and $y$ lie in different components of $M\setminus\{z\}$. Assume $x$ is not a global cut point, then the {\it block} containing $x$ is the collection of points $y\in M$ such that $xRy$, and will be denoted $[x]$. We make the exception that any singleton set satisfying these conditions will not be considered a block. If two blocks $[u]$ and $[v]$ intersect, then they intersect in an element of $P$ or $[u]=[v]$ (see \cite{Swe99}).

If $M$ is the boundary of a relatively hyperbolic group with tame peripherals, then $M$ is a Peano continuum and the relation $R$ naturally associates to $M$ a simplicial bipartite tree $T$ \cite{Bow01}. The vertices of $T$ correspond to elements of $\Pi$ and the set of blocks $B$. Additionally, two vertices $b\in B$ and $p\in \Pi$ are adjacent if $p\in b$.

Now, let $T$ be the Bass-Serre tree for the maximal peripheral splitting $\mathcal{G}$ of $G$ (see Theorem \ref{theorem:Bowditch Accessibility}), and assume that $\R$ and $\sP$ are the collections of component and peripheral vertices respectively. Let $v\in\sP$. A subtree $S$ of $T$ is a {\it branch rooted at} $v$ if it is the closure of a component of $T\setminus\{p\}$. Bowditch \cite{Bow01} has shown the following:

\begin{theorem}
\label{theorem:treelike structure}
Let $(G,\bbp)$ be relatively hyperbolic with tame peripherals and connected Bowditch boundary. Assume that $T$, $\R$, and $\sP$ are as above. There exists an injective map $\beta\colon\sP\cup\bndry T\rightarrow\relbndryG$ and for every $v\in\R$ there exists a unique set $B(v)\subset\relbndryG$ satisfying the following:
\begin{enumerate}
\item $B(v)$ is a proper subcontinuum of $\relbndryG$ for every $v\in\R$, which contains a point not in the image of $\beta$. Additionally, if $u,v\in\R$ are distinct and $B(u)\cap B(v)\neq\emptyset$, then $B(u)\cap B(v)=\bigl\{\beta(p)\bigr\}$ for some $p\in\sP$ adjacent to both $u$ and $v$.
\item If $x\in\sP$ then $\beta(x)$ is a parabolic point.
\item If $(x_n)\subset \sP$ is a sequence of points converging to $i\in\bndry T$, then the sequence $\beta(x_n)$ converges to a point $\iota=\beta(i)$ in $\relbndryG$. Such a point will be referred to as an {\it ideal} point.

\item If $v$ is a vertex in $\R$ then $B(v)$ is block which cannot be separated by a cut point. If $H=\Stab_G(R)$, then the action of $H$ on $B(v)$ is geometrically finite with maximal parabolic subgroups
\[\bbq=\bigset{Q}{Q=stab_G(v)\cap P\spc\ {\it with}\spc  Q\spc {\it infinite} \spc {\it and}\spc P\in\bbp}.\]
Consequently, $(H,\bbq)$ is relatively hyperbolic with $\blockbndry=B(v)$. Additionally, $\blockbndry$ is locally connected (See \cite{Bow99b}).

\item Given a subtree $S$ in $T$ and let $\sP(S)$ and $\R(S)$ be $\sP\cap S$ and $\R\cap S$, respectively. Then the set $\Psi^0(S)=\beta\big(\sP(S)\big)\cup\bigcup_{v\in\R(S)}B(v)$ is connected and its closure is the set $\Psi(S)=\beta\big(\sP (S)\cup\bndry S\big)\cup\bigcup_{v\in\R(S)}B(v)$. If $S$ is a branch in $T$ rooted at $v\in\sP$, then $\Psi(S)$ is called a {\it branch of} $\relbndryG$ {\it rooted at} $\beta(v)$.
\item $\Psi(T)=\relbndryG$
\item If $v$ is a vertex in $\R$, then $B(v)$ does not contain any {\it ideal} points.
\item Every ideal point $\iota$ has a neighborhood base consisting of branches, and any branch containing $\iota$ is a neighborhood of $\iota$.
\item Let $\beta^*\colon(\sP\cup\bndry T)\cup\R\rightarrow\relbndryG$ be the multi-valued map defined by $\beta^*(v)=\beta(v)$ for every $v\in\sP\cup\bndry T$ and $\beta^*(v)=B(v)$ for any $v\in\R$. Then $\beta^*$ is $G$-equivariant.
\end{enumerate}
\end{theorem}

\medskip

\begin{corollary}
\label{label:ideal point}
A local cut point in $\relbndryG$ must be in a block, i.e. ideal points are not local cut points.
\end{corollary}

\begin{proof} Let $\iota$ be an ideal point in $\relbndryG$. Then $\iota$ is contained in some branch, $\Psi(B)$. We first show that $\Psi(B)\setminus\{i\}$ connected. We have that $\Psi^0(B)\subset\Psi(B)\setminus\{\iota\}\subset\Psi(B)$, the set $\Psi^0(B)$ is connected, and we have that $\Psi(B)$ is the closure of $\Psi^0(B)$. So, $\Psi(B)\setminus\{\iota\}$ must be connected. Thus $\relbndryG\setminus\{\iota\}$ is connected.

Now if $U$ is any neighborhood of $\iota$, we have from Theorem \ref{theorem:treelike structure}(8) that there is branch $B\subset U$ containing $\iota$. By the argument in the preceding paragraph $B\setminus\{\iota\}$ is connected and $\iota$ cannot be a local cut point (see Section \ref{subsec:Cut Point Structures In Metric Spaces}).
\end{proof}

\medskip
The following theorem was communicated to the author by Chris Hruska and relies on Theorem \ref{theorem:treelike structure}(4) and known results about the action of the $G$ on $\relbndryG$. In particular, Bowditch has shown \cite{Bow1} that the action of $G$ on $\relbndryG$ is {\it minimal}, i.e. $\relbndryG$ does not properly contain a closed $G$-invariant subset. Because it will be of use in Section \ref{section:classification}, it is worth noting that the action of $G$ on $\relbndryG$ is minimal if and only if $\Orb_G(m)$ is dense for every $m\in\relbndryG$.

\begin{theorem}
\label{theorem :non-trivial splitting}
Assume $(G,\bbp)$ is relatively hyperbolic with tame peripherals, $\relbndryG$ is connected, and $\relbndryG$ contains a global cut point. Then $(G,\bbp)$ splits non-trivially over each edge group in the maximal peripheral splitting of $(G,\bbp)$ that corresponds to an edge connecting a component vertex to a peripheral cut point vertex.
\end{theorem}

\begin{proof}Assume that $T$ is the Bass-Serre tree for the maximal peripheral splitting of $G$. Assume there exists an edge $e$ in $T$ connecting a component vertex $c$ to a peripheral cut point vertex $p$, also assume that $G$ does not split over the edge group $G_e$ non-trivially. Then there is a $G$-invariant subtree $B$ in $T$ which does not contain $e$ (see \cite{HR1} Lemma 12.8). As a cut point vertex, $p$ is adjacent to at least two component vertices. Because $e\not\subset B$, there is at least one component vertex $u$ which is not in $B$. Then by Theorem \ref{theorem:treelike structure}(1) there is a block $B(u)$ which is not entirely contained in $\Psi(B)$. Thus $\Psi(B)\neq\Psi(T)$. By Theorem \ref{theorem:treelike structure}(9) the map $\beta^*$ is $G$-equivariant, so $\Psi(B)$ is a closed $G$-invariant proper subspace $\relbndryG$. This implies that the action of $G$ on $\relbndryG$ is not minimal, a contradiction.
\end{proof}

\medskip
As a corollary we have:

\begin{corollary}
\label{corollary:theorem :non-trivial splitting}
Assume $(G,\bbp)$ is relatively hyperbolic with tame peripherals, $\relbndryG$ is connected, and let $T=\R\sqcup\sP$ be the Bass-Serre tree for the maximal peripheral splitting of $(G,\bbp)$. Suppose that $p\in\sP$ is a cut point vertex of $T$, and set $P=\Stab_G(p)$. If $H=\Stab_G(v)$ for some $v\in\R$ which is adjacent to $p$, then $G$ splits non-trivially relative to $\bbp$ over an infinite subgroup $G_e$ of $P\cap H$.
\end{corollary}

\begin{proof}
Theorem \ref{theorem:treelike structure}(2) gives that $\beta(p)$ is a parabolic point of $\relbndryG$. By hypothesis $p$ separates $T$ into at least two components, Theorem \ref{theorem:treelike structure}(5) gives that $\beta(p)$ is a cut point of $\relbndryG$. The result follows from Theorem \ref{theorem :non-trivial splitting} and Proposition 10.1 of \cite{Bow1}, which says that $\relbndryG$ is connected if and only if $G$ does not split non-trivially over any finite subgroup relative to $\bbp$.
\end{proof}

\subsection{Decompositions and reduction}
A {\it decomposition} $\D$ of a topological space $X$ is a partition of $X$. Associated to $\D$ is the {\it decomposition} {\it space} whose underlying point set is $\D$, but denoted $X/\D$. The topology of $X/\D$ is given by the {\it decomposition map} $\pi\colon X\rightarrow X/\D$, with $x\mapsto D$, and where $D\in\D$ is the unique element of the decomposition containing $x$. A set $U$ in $X/\D$ is deemed open if and only if $\pi^{-1}(U)$ is open in $X$. A subset $A$ of $X$ is called {\it saturated} (or $\D$-saturated) if $\pi^{-1}\big(\pi(A)\big)=A$. The {\it saturation} $\Sat(A)$ of $A$ is the union of $A$ with all $D\in\D$ that intersect $A$. The decomposition $\D$ is said to be {\it upper semi-continuous} if every $D\in\D$ is closed and compact, and for every open set $U$ containing $D$ there exists and open set $V\subset U$ such that $D\subset V$ and $\Sat(V)$ is contained in $U$. In Proposition I.3.1 of \cite{Daverman} Daverman shows that the decomposition map of an upper semi-continuous decomposition is proper. It is then an easy corollary that in an upper semi-continuous decomposition the saturation of a compact set is compact. An upper semi-continuous decomposition $\D$ is called {\it monotone} if the elements of $\D$ are connected. The following is a key characteristic of monotone decompositions and may be found as Proposition I.4.1 of \cite{Daverman}:

\begin{proposition}
\label{prop: daverman monotoneconnected}
Let $\D$ be a decomposition of a space $X$. Then $\D$ is monotone if and only if $\pi^{-1}(A)$ is connected for every connected subset $A$ of $X\setminus\D$.
\end{proposition}

A collection of subsets $\mathcal{S}$ of a metric space is called a {\it null family} if for every $\epsilon>0$ there are only finitely $S\in\mathcal{S}$ with $diam(S)>\epsilon$. The following proposition can be found as Proposition I.2.3 in \cite{Daverman}.

 \begin{proposition}
 \label{proposition:null fam uppersemi}
 Let $\mathcal{S}$ be a null family of closed disjoint subsets of a compact metric space $X$. Then the associated decomposition of $X$ is upper semi-continuous.
 \end{proposition}

\begin{lemma}
\label{lemma:sat of closed}
If $\D$ is an upper semi-continuous decomposition of a space $X$, then the saturation of a closed set is closed.
\end{lemma}

Lemma \ref{lemma:sat of closed} follows from Proposition I.1.1 of \cite{Daverman}.

\begin{lemma}
\label{lemma:fatigued quotient}
If $\D$ is an upper semi-continuous 
decomposition of a generalized Peano continuum $X$, then $X/\D$ is a generalized Peano continuum.
\end{lemma}

\begin{proof}
Let $Y=X/\D$. We want that $Y$ is connected, locally connected, locally compact, $\sigma$-compact, and metrizable. Clearly, $Y$ is connected. By Theorem 27.12 of \cite{Wil1} the quotient of a locally connected space is locally connected.

To prove local compactness note that Lemma \ref{lemma:sat of closed} implies the quotient map $f\colon X\rightarrow Y$ is closed. The image of a locally compact space under a closed map is locally compact provided the preimage of each point is compact (see \cite{Wil1} Exercise 18C.2). Thus $Y$ is locally compact, because the elements of $\D$ are compact.

We still require that $Y$ is $\sigma$-compact and metrizable. The continuous image of a $\sigma$-compact space is $\sigma$-compact, so $Y$ is $\sigma$-compact. Proposition I.2.2 of \cite{Daverman} gives that the image of a  metric space under an upper semi-continuous decomposition is metrizable. Thus $Y$ is a generalized Peano continuum.
\end{proof}

\begin{lemma}
\label{lemma:Ends of quotient}
 Let $X$ be a generalized continuum. Assume that $\D$ is a monotone upper semi-continuous decomposition, and let $f\colon X\rightarrow X/\D$ be the decomposition map. If $Q=X/\D$, then $f$ induces a homeomorphism between $\Ends(X)$ and $\Ends(Q)$.
\end{lemma}

\begin{proof}
By Proposition I.3.1 of \cite{Daverman} the decomposition map of an upper semi-continuous decomposition is proper. So, by Proposition \ref{proposition:Craig's prop} we have that $f$ can be continuously extended to a map $\hat{f}\colon X\cup\Ends(X)\rightarrow Q\cup\Ends(Q)$. We only need that the restriction $f^*\colon\Ends(X)\rightarrow\Ends(Q)$ is a bijection.

Let $(C_1,C_2,C_3,..)$ be an exhaustion of $X$. The elements of $\D$ are compact and we have that the saturation of a compact set is compact. So, we may assume that each $C_i$ is saturated. We first show that $f^*$ is surjective. The sequence $\big(f(C_i)\big)_{i=1}^{\infty}=\big(f(C_1),f(C_2),f(C_3),...\big)$ is an exhaustion of $Q$ and $\Ends(Q)$ is independent of choice of exhaustion, so let $(A_1,A_2,A_3,...)\in\Ends(Q)$ be defined using the exhaustion $\big(f(C_i)\big)$. For each $i$  the pre-image $f^{-1}(A_i)$ is contained $f^{-1}\big(Q\setminus f(C_i)\big)= X\setminus C_i$, and by monotonicity $f^{-1}(A_i)$ is connected by Proposition \ref{prop: daverman monotoneconnected}. Since, $f^{-1}(A_1)\supset f^{-1}(A_2)\supset f^{-1}(A_3)\supset\cdots$, we have that $\big(f^{-1}(A_i)\big)_{i=1}^{\infty}$ is an end of $\Ends(X)$.

Now, let $(E_i)_{i=1}^{\infty}$ and $(F_i)_{i=1}^{\infty}$ be distinct elements of $\Ends(X)$. Then there exists an $i\in\bbn$ such that $E_j\neq F_j$ for all $j\geq i$. Because the $C_i$ are saturated, monotonicity implies $f(E_j)\neq f(F_j)$ for all $j\geq i$. Thus $\big(f(E_n)\big)_{n=1}^{\infty}$ and $\big(f(F_n)\big)_{n=1}^{\infty}$ are distinct.
\end{proof}

\begin{corollary}
\label{corollary:image of local cut point}
Assume that $\D$ is a monotone upper semi-continuous decomposition, let $f\colon X\rightarrow X/\D$ be the decomposition, and let $x$ be a point of $X$ such that $\{x\}\in\D$. If $x$ is local cut point which is not a global cut point, then $f(x)$ is a local cut point and $val(x)=val\big(f(x)\big)$.
\end{corollary}

\medskip
Returning to the setting of Bowditch boundaries we will use the notation introduced in Section \ref{subsection: Blocks and Branches}. Let $r$ be an element of $\R$. Any branch not containing  the block $B(r)$, but rooted at a point in the block $B(r)$ is said to be {\it attached to} $B(r)$. The union of all branches attached to $B(r)$ with common root will be called a {\it full branch attached to} $B(r)$.

\begin{lemma}
\label{lemma:full branches are null}
Let $R=B(v)$ for some $v\in \R$. The collection of full branches attached to $R$ forms a null family of disjoint connected closed sets.
\end{lemma}

\begin{proof}
Let $F$ be a full branch attached to $R$ with root $\beta(p)$ for some $p\in\sP$. Then $F$ is the subcontinuum associated by $\Psi$ to the subtree $S$ of $T$ consisting of all branches in $T$ rooted at $p$ which do not contain the vertex  $v$. By Theorem \ref{theorem:treelike structure} $F$ is connected and closed.

Let $F'$ be another full branch attached to $R$, and assume that $F'$ is rooted at $\beta(q)$ for some $q\neq p$. If $S'$ is the subtree of $T$ of all branches in $T$ rooted $q$ and not containing $v$, then $\Psi(S')=F'$. As $S\cap S'=\emptyset$ we have that $F\cap F'=\emptyset$, because the map $\beta$ is injective and that blocks associated to component vertices are unique.

Now, let $\epsilon>0$. Bowditch has shown in section 8 of \cite{Bow01} that the set of all branches attached to a component of $\relbndryG$ forms a null family, so there are only finitely many branches of diameter $\epsilon/2$. Let $x_1$ and $x_2$ be two points in a full branch $D$, and let $B_1\subset D$ and $B_2\subset D$ be branches containing $x_1$ and $x_2$, respectively. The distance between $x_1$ and $x_2$ is at most $diam(B_1)+diam(B_2)$. If there were infinitely many full branches of diameter greater than epsilon, then there would be infinitely may branches of diameter greater than $\epsilon/2$, a contradiction.
\end{proof}

It follows from Lemma \ref{lemma:full branches are null} that:

\begin{lemma}
\label{lemma: branches uppersemi}
 Let $R=B(v)$ for some $v\in \R$, and define $f\colon\relbndryG\rightarrow R$ to be the quotient map obtained by identifying all full branches attached to $R$ with their roots. Then $f$ is an upper semi-continuous monotone retraction onto $R$.
\end{lemma}

\begin{lemma}
\label{lemma:image of conical limit local cut point}
Let $x$ be a point contained in a block $R$. The point $x$ is a local cut point of $\relbndryG$ and a conical limit point of the action of $G$ on $\relbndryG$ if and only if $f(x)$ is a local cut point of $R$ and a conical limit point of the action of $\Stab_G(R)$ on $R$.
\end{lemma}

\begin{proof} Notice that $x$ is not contained in a full branch attached to $R$, so $\{x\}$ is an element of the decomposition which is not a global cut point. Lemma \ref{lemma: branches uppersemi} gives that $f$ is an upper semi-continuous decomposition. Then Corollary \ref{corollary:image of local cut point} implies that $f(x)$ is a local cut point of $R$. By Proposition \ref{proposition:Hruska Thm 9.1} $f(x)$ is conical limit point of the action of $\Stab_G(R)$ on $R$.
The reverse direction follows from Lemma \ref{lemma: branches uppersemi}, Proposition \ref{proposition:Hruska Thm 9.1}, and the observation that if $f(x)$ is a local cut point, then $\big|\Ends(R\setminus\{f(x)\})\big|\geq 2$ and Lemma \ref{lemma:Ends of quotient} implies $\big|\Ends(\relbndryG\setminus\{x\})\big|\geq 2$. \end{proof}

\begin{lemma}
\label{lemma:cutpair to cut pair}
Let $\{x,y\}\subset R$ and assume that $x$ and $y$ are not parabolic points. Then $\{x,y\}$ is a cut pair in $\relbndryG$ if and only if $\bigl\{f(x),f(y)\bigr\}$ is a cut pair in $R$. Moreover, $f$ induces a bijection between components of $\relbndryG\setminus\{x,y\}$ and components of $R\setminus\bigl\{f(x),f(y)\bigr\}$.
\end{lemma}

\begin{proof}
This result follows from Proposition \ref{prop: daverman monotoneconnected}. We first show that there is a bijection between components of $\relbndryG\setminus\{x,y\}$ and components of $R\setminus\bigl\{f(x),f(y)\bigr\}$.  Assume $\bigl\{f(x),f(y)\bigr\}$ is a cut pair in $R$. Because $\{x\},\{y\}\in\D$ the preimage of each component of $R\setminus\bigl\{f(x),f(y)\bigr\}$ is a saturated connected set in $\relbndryG\setminus\{x,y\}$. Because $\D$ is monotone the connected components of $\relbndryG\setminus\{x,y\}$ are saturated, and therefore not identified under $f$. Thus we have a bijection between components of $\relbndryG\setminus\{x,y\}$ and components of $R\setminus\bigl\{f(x),f(y)\bigr\}$.

Now, assume $\{x,y\}$ is a cut pair. Then $\relbndryG\setminus\{x,y\}$ has at least two components and the above implies that $R\setminus\bigl\{f(x),f(y\bigr)\}$ has at least two components. If $\bigl\{f(x),f(y)\bigr\}$ is a cut pair, then again the result follows from the above.
\end{proof}

By Proposition \ref{proposition:Hruska Thm 9.1} a conical limit point of the action of $\Stab_G(R)$ on $R$ is a conical limit point of the action of $G$ on $\relbndryG$, so as a corollary of Lemma \ref{lemma:cutpair to cut pair} we obtain:

\begin{corollary}
\label{corollary:sim in R equals sim in boundary}
Let $x,y\in R$ be conical limit points of the action of $\Stab_G(R)$ on $R$. Then $x\sim y$ in $\relbndryG$ if and only if $f(x)\sim f(y)$ in $R$.
\end{corollary}

\begin{lemma}
\label{lemma:HvG-translate inseparable}
Assume $R=B(v)$ for some $v\in\R\subset\relbndryG$ and let $H=\Stab_G(R)$. A cut pair $\{a,b\}$ in $R$ is $H$-translate inseparable if and only if it is $G$-translate inseparable.
\end{lemma}

\begin{proof}
Notice that by Lemma \ref{lemma:cutpair to cut pair} we have that $\{a,b\}$ is a cut pair of $\relbndryG$. If $h\in H$, then by Lemma \ref{lemma:cutpair to cut pair} a translate $\{ha,hb\}$ separates $\{a,b\}$ in $R$ if and only if $\{ha,hb\}$ separates $\{a,b\}$ in $\relbndryG$. So, if $\{a,b\}$ is $G$-translate inseparable, then $\{a,b\}$ is $H$-translate inseparable.

Now assume that $\{a,b\}$ is $H$-translate inseparable and let $g\in G$. Because cut pairs cannot be separated by cut points, the pair $\{ga,gb\}$ must be in $R$ or $\relbndryG\setminus R$. If $\{ga,gb\}$ is in $\relbndryG\setminus R$ then $\{ga,gb\}$ cannot separate $\{a,b\}$. If $\{ga,gb\}$ is in $R$ then $g$ must be in $H$. That $g\in H$ follows from the definition of a block as a maximal set of points of $\relbndryG$ which cannot be separated by cut points and the fact that $g$ is a homeomorphism. For if $g$ sent a point $x\in R\setminus\{a,b\}$ to a point  of $\relbndryG\setminus R$, then by the definition of $R$ there must exist a cut point of $\relbndryG$ which separates either the pair $\{gx,ga\}$ or the pair $\{gx,gb\}$. Thus $gR$ contains points which can be separated by a cut point, but $g$ is a homeomorphism. So, $R$ contains points which can be separated by a cut point, a contradiction. Thus $\{a,b\}$ is $H$-translate inseparable, then $\{a,b\}$ is $G$-translate inseparable.
\end{proof}

\begin{corollary}
\label{corollary:cut pair in block}
Let $R=B(v)$ for some $v\in\R$. If $\{a,b\}$ is a cut pair of $R$ which is $\Stab_G(R)$-translate inseparable and does not contain any parabolic points, then $\{a,b\}$ is $G$-translate inseparable cut pair of $\relbndryG$ which does not contain any parabolic points. Additionally, if $R$ contains a necklace $\nu$ and $i\colon R\hookrightarrow\relbndryG$ is the inclusion map, then $i(\nu)$ is a necklace in $\relbndryG$.
\end{corollary}

\begin{proof}
The first result follows from Lemma \ref{lemma:HvG-translate inseparable}.

Let $\nu$ be a necklace in $R$. Then $\nu$ is the closure of an $\sim$-class which is contained in $R$. Call this class $C$. Because cut pairs in $\relbndryG$ cannot be separated by global cut points and the elements of $C$ are conical limit points, Corollary \ref{corollary:sim in R equals sim in boundary} implies that $i(C)$ is a $\sim$-class in $\relbndryG$. $R$ is closed, so the inclusion map is closed. Thus we have $i(\nu)=i(\overline{C})=\overline{i(C)}$, which is a necklace in $\relbndryG$.
\end{proof}

\section{Local cut points in $\relbndryG$}
\label{section:Local Cut Points in relbndry}

The goal of this section is to describe the ways local cut points occur in $\relbndryG$ (see Theorem \ref{theorem:Collection thm}). In the hyperbolic setting Bowditch \cite{Bow98} showed that a local cut point must be contained in a translate inseparable loxodromic cut pair or a necklace. We wish to adapt Bowditch's result to the relatively hyperbolic setting (see Theorem \ref{theorem:collecting conical}). As first observed by Guralnik \cite{Gur}, a careful examination of \cite{Bow98} reveals that much of Bowditch's argument in \cite{Bow98} could directly translate to the Bowditch boundary $\relbndryG$ if one only considers local cut points which are conical limit points and assumes that $\relbndryG$ has no global cut points. However, there are two key steps (Lemma 5.25 and Lemma 5.17 of \cite{Bow98}) in the proof where Bowditch depends heavily hyperbolicity. Namely, in Section 5 of \cite{Bow98} Bowditch requires that $G$ act as a uniform convergence group on its boundary, i.e that the action on the triple space is proper and cocompact. As mentioned in Section \ref{section:Preliminaries} in the relatively hyperbolic setting the action of $G$ on $\relbndryG$ is not uniform in general.  Lemma \ref{lemma:control theorem} and Proposition \ref{prop:Key prop} generalize the critical steps (Lemma 5.2 and Lemma 5.17 of \cite{Bow98}) of Bowditch's argument to the relatively hyperbolic setting and allow us to use Bowditch's results concerning local cut points to prove the main result of this section (Theorem \ref{theorem:Collection thm}). We remark that Lemma \ref{lemma:control theorem} is also proved in \cite{Gur}, but for completeness we include an alternate more self-contained proof, which uses different techniques.

\medskip

\subsection{A key lemma}
\label{subsec:first key lemma}
In this section we prove the following technical lemma:

\medskip
\begin{lemma}
\label{lemma:control theorem}
Let $(G,\bbp)$ be a relatively hyperbolic group. There exist finite collections $(U_i)_{i=1}^p$ and $(V_i)_{i=1}^p$ of open connected sets of $\relbndryG $ with $\overline{U}_i\cap\overline{V}_i=\emptyset$, and such that if $K\subseteq \relbndryG $ is closed and $x\in \relbndryG\backslash K$ is a conical limit point then there exists $g\in G$ and $i\in\{1,...,p\}$ such that $gx\in U_i$ and $gK\subseteq V_i$.
\end{lemma}

\medskip

 We postpone the proof of Lemma \ref{lemma:control theorem}, as it will require a few lemmas.  Let $X$ be the proper $\delta$-hyperbolic space on which $G$ acts as given by the definition of relatively hyperbolic (see Section \ref{section:Preliminaries}). We know from Theorem \ref{theorem:finitely many orbits of maximal parabolic} that there finitely many orbits of horoballs in $\mathcal{B}$. Let $B_1,B_2,...,B_n$ be representatives from each orbit and $p_1,p_2,...,p_n$ the associated parabolic points for each representative horoball. In \cite{Bow1} it is shown that $C_i=\Fr(B_i)/Stab_G({p_i})$ is compact for every $i\in\{1,2,...,n\}$ and from the definition of relatively hyperbolic we know that $G$ acts cocompactly on $(X\setminus\mathcal{B})$. Let $A$ be the fundamental domain of the action of $G$ on $(X\setminus\mathcal{B})$, and define
 \[C=A\cup C_1 \cup C_2\cup \cdots\cup C_n.\]
Then $C$ is a compact subset of $X$ and $\Orb_G(C)\supseteq X\setminus\mathcal{B}$.

\medskip

Let $\Theta_2\relbndryG$ the space of distinct pairs in $\relbndryG$ and define $E(C)\subseteq\pairs$ to be the collection of pairs $(x,y)$ such that $x=c(\infty)$ and $y=c(-\infty)$ for some line $c\colon\bbr\rightarrow X$ with $\im(c)\cap C\neq\emptyset$.

\begin{lemma}
\label{lemma:pairs compact}
The set $E(C)$ is compact in $\Theta_2\relbndryG$.
\end{lemma}

\medskip

The proof of Lemma \ref{lemma:pairs compact} follows from the fact that for any pair $(x,y)\in\Theta_2\relbndryG$ we may find a line whose ends are $x$ and $y$ (see \cite{BH1} Chapter III ). Then sequential compactness and a standard diagonal argument show that a sequence of lines each meeting $C$ converges to a line meeting $C$. We leave the details as an exercise.

\medskip
The action of $G$ is by isometries, the translates of $C$ cover $X\setminus\B$, and a line $\ell$ cannot be completely contained in a horoball. So, there must be a $g\in G$ and a line $\ell'$ such that $gC\cap \ell\neq\emptyset$, $\ell'\cap C\neq\emptyset$, and $g\ell'=\ell$.   Consequently, we obtain:
\medskip

\begin{corollary}[Tukia, Gerasimov]
\label{corollary: cocompact on pair}
$G$ acts cocompactly on $\Theta_2\relbndryG$.
\end{corollary}

Corollary \ref{corollary: cocompact on pair} was first observed by Gerasimov \cite{Ger09} following results of Tukia \cite{Tukia98}.

\begin{lemma}
\label{lemma:neighborhood pair}
There exist finite collections $(U_i)_{i=1}^p$ and $(V_i)_{i=1}^p$ such that $\overline{U}_i\cap\overline{V}_i=\emptyset$ for every $i\in\{1,...,p\}$, and such that if $x,y\in\relbndryG$ then there exists $g\in G$ and $i\in\{1,...,p\}$ with $gx\in U_i$ and $gy\in V_i$.
\end{lemma}

\begin{proof} Let $d$ be a metric on $\relbndryG$, and let $K$ be a compact set whose $G$ translates cover $\Theta_2\relbndryG$. Clearly the intersection of $K$ with the diagonal of $\relbndryG\times\relbndryG$ is empty. For every $(x,y)\in K$ define $r(x,y)=\frac{1}{4}d(x,y)$ and define $U_x=B\big(x,r(x,y)\big)$ and $V_y=B\big(y,r(x,y)\big)$. Then $\bigcup_{(x,y)\in C}(U_x\times U_y)$ covers $K$. By compactness there exist finitely many $(x_i,y_i)\in K$ such that $U_{x_i}\times V_{x_i}$ cover $K$. Notice that by construction $\overline{U}_{x_i}\cap \overline{V}_{y_i}=\emptyset$. Thus by the cocompactness of the action we are done.\end{proof}

\begin{proof}[Proof of Lemma \ref{lemma:control theorem}] Let $x$ be a conical limit point. By the definition of conical limit point there exists $(g_n)\in G$ and distinct points $\alpha,\beta\in\relbndryG$ such that $g_nx\rightarrow \alpha$ and $g_ny\rightarrow\beta$ for every $y\in\relbndryG\setminus\{x\}$; moreover, by passing to a subsequence we may assume that the members of the sequence $(g_n)$ are distinct.

Because $G$ acts on $\relbndryG$ as a convergence group, we have that every sequence $(g_n)$ of distinct group elements has a subsequence $(g_i)$ such that if $K\subset\relbndryG\setminus \{x\}$ then for any neighborhood $V\ni\beta$ there exists $g_{i_0}\in(g_i)$ such that $g_{i_0}\in V$.

Let $(U'_i)_{i=1}^p$ and $(V'_i)_{i=1}^p$ be the neighborhoods found in Lemma \ref{lemma:neighborhood pair}. As $(\alpha,\beta)\in\Theta_2\relbndryG$ there exists $g\in G$ and $i\in\{1,...,p\}$ such that $g\alpha\in U'_i$ and $g\beta\in V'_i$. Set $U_i=g^{-1}U'_i$ and $V_i=g^{-1}V'_i$. Then for large enough $n$ we have $g_nx\in U_i$ and $g_nK\subseteq V_i$.\end{proof}

\subsection{The stabilizer of a necklace is relatively quasiconvex}
\label{subsec:rel qc}

The goal of this section is to prove the following proposition:

\begin{proposition}
\label{prop:Key prop}
Suppose $(G,\bbp)$ is relatively hyperbolic with tame peripherals, and set $M=\relbndryG$. Assume $M$ is connected, has no global cut points, not homeomorphic to $S^1$, and $\nu$ is a necklace in $M$. Then $\Stab_G(\nu)$ is relatively quasiconvex, acts minimally on $\nu$, and any jump in $\nu$ is a loxodromic cut pair and translate inseparable.
\end{proposition}

The proof of Proposition \ref{prop:Key prop} is divided into the following series of lemmas and corollaries.

A collection of subspaces $\A$ of a metric space $Z$ is called {\it locally finite} if only finitely many members of $\A$ intersect any compact set $K\subset Z$. To prove Proposition \ref{prop:Key prop} we will use the following lemma twice. The following may be found as Proposition 7.2 of \cite{HR1}.

\begin{proposition}
\label{proposition:lclfinaction}
Let $G$ be a group acting properly and cocompactly on a metric space $Z$, and let $\A$ be a locally finite collection of closed  subspaces of $Z$. Then $\Stab(A)$ acts cocompactly on $A$ for every $A\in\A$, and the elements of $\mathcal{A}$ lie in finitely many $G$-orbits.
\end{proposition}

Until otherwise stated we will assume the following hypotheses in this section. Let $(G,\bbp)$, $M$, and $\nu$ be as in the statement of Proposition \ref{prop:Key prop}. Let $X$ be a $\delta$-hyperbolic space on which $G$ acts with a cusp uniform action (see Section \ref{subsec:relhypgrps}). Let $\B$ be a $G$-equivariant family of open horoballs based at the parabolic points of $\bndry X$ as given by the definition of relatively hyperbolic, and let $K$ be a compact set such that $X\setminus\mathcal{B}\subset\Orb_G(K)$.

The following lemma was observed by Gromov in Section 7.5.A of \cite{Gro87}. The following proof is based on an argument due to Dahmani (see Proposition 1.8 of \cite{Dahm03}).

\begin{lemma}
\label{lemma:actually the limit set}
The limit set $\Lambda(\Jn(\nu))$ is equal to $\nu$.
\end{lemma}

Before we prove Lemma \ref{lemma:actually the limit set}, we recall some basic definition from the theory of hyperbolic groups (see for example \cite{Gro87} or \cite{GH90}). Let $z$ be a base point of $X$. The {\it Gromov product} of $x,y\in X$ with respect to $z$ is defined to be \[(x|y)_z=\frac{1}{2}\bigl\{d(x,z)+d(y,z)-d(x,y)\bigr\}\]. The Gromov product is extended to $X\cup\bndry X$ by setting \[(x|y)_z=\sup\liminf\limits_{i,j\rightarrow 0}(x_i|x_j)\] where the supremum is taken over all sequences $(x_i)\rightarrow x$ and $(y_j)\rightarrow y$. The Gromov product measures the distance from the point $z$ to the geodesic between $x$ and $y$ up to finite error (see for example \cite{BH1} Definition III.H.1.19 and Exercise III.H.3.18(3)), in other words we have:

\begin{lemma}
\label{lemma:inner product v distance}
There is a constant $\rho$ so that $\big|(x|y)_z-d(z,[x,y])\big|<\rho$ for all $z\in X$ and all $x,y\in X\cup\bndry X$.
\end{lemma}

Also, by \cite{BH1} Remark III.H.3.17(6) we have:

\begin{lemma}
\label{lemma:conv seq equal GP to infty}
If $(a_i)$ is a sequence of points in $X\cap\bndry X$ and $a\in\bndry X$, then $(a_i)\rightarrow a$ if and only if $(a_i|a)\rightarrow\infty$ .
\end{lemma}

\begin{proof}[Proof of Lemma \ref{lemma:actually the limit set}]
Clearly $\nu\subset\Lambda(\Jn(\nu))$. Let $(x_n)$ be a sequence of points from distinct lines in $\Jn(\nu)$, and converging to a point $x$ in $\bndry X$. For each $i$ let $c_i$ be the line containing $x_i$, and set $a_i=c_i(\infty)$ and $b_i=c_i(-\infty)$. The pairs $\{a_i,b_i\}$ form a sequence in $\Theta_2\nu$. As $\Theta_2\nu$ is a compact subset of $\Theta_2\relbndryG$, by passing to a subsequence if necessary we may assume $\big(\{a_i,b_i\}\big)$ converges to some pair $\{a,b\}\in\Theta_2\nu$. We claim that $x=a$ or $x=b$.

Let $z$ be a base point for $X$. For each $i$ let $y_i$ denote the point of $c_i$ nearest $z$. The points $y_i$ divide the lines $c_i$ into two sides $A_i$ and $B_i$ with $A_i\cap B_i=\{y_i\}$, $a_i\in A_i$, and $b_i\in B_i$. The infinitely many members of the sequence $(x_i)$ must be in one of the collections $\{A_i\}$ or $\{B_i\}$. We may assume without loss of generality that infinitely many members of the sequence $(x_i)$ are in $\{B_i\}$. Let $t_i\in [x_i,b_i]\subset B_i$. Then we have that $d(z,t_i)\geq d(z,x_i)-2\delta$ taking the minimum over all $t_i$ we have that $d\big(z,[x_i,b_i]\big)\geq d(z,x_i)-2\delta$. Thus by Lemma \ref{lemma:inner product v distance} $(x_i|b_i)_z\geq d(z,x_i)-2\delta-\rho$.

Now, as $(b_i)\rightarrow b$ we have that $(b_i|b)\rightarrow\infty$. By $\delta$-hyperbolicity of $X$ we have $(x|b)\geq \min\big\{(x_i|b_i)_z,(b_i|b)_z\big\}-\delta'$ where $\delta'$ is some multiple of $\delta$. Thus $(x|b)_z\rightarrow\infty$, which implies that $(x_i)\rightarrow b$.
\end{proof}

\begin{lemma}
\label{lemma:necklaces are locally finite}
Let $\mathcal{J}=\bigset{\Jn(\nu)}{\nu \text{ is a necklace in } \bndry X}$. $\mathcal{J}$ is locally finite in $X$.
\end{lemma}

\begin{proof}
Let $C$ be a compact subset in $X$. Lemma \ref{lemma:pairs compact} implies that end points of the set of all lines intersecting $C$ gives us a compact subset $E(C)$ of $\Theta_2\relbndryG$. If $J=\Jn(\nu)\in\mathcal{J}$ and $J\cap C$ is non-empty, then there exists a line $\ell$  contained in $J$ such that $\bigl\{\ell(\infty),\ell(-\infty)\bigr\}\subset\nu$ and $\bigl\{\ell(\infty),\ell(-\infty)\bigr\}\in E(C)$. Lemma 3.16 of \cite{Bow98} Bowditch shows that a compact set of $\Theta_2\relbndryG$ can only intersect finitely many $\sim$-classes, thus only finitely many members of $\mathcal{J}$ can intersect $C$.
\end{proof}

For the remainder of this section define $H=\Stab_G(\Jn(\nu))$ for some fixed necklace $\nu$.

\begin{lemma}
\label{lemma:cocompact on truncation}
The subgroup $H$ acts properly and cocompactly on $\Jn(\nu)\cap(X\setminus\B)$.
\end{lemma}

\begin{proof}
Lemma \ref{lemma:necklaces are locally finite} implies that  $\orb_G\big(\Jn(\nu)\big)$ is locally finite in $X$. Thus, only finitely many members of $\orb_G\big(\Jn(\nu)\big)$ intersect any compact set $C\subset \Jn(\nu)\cap(X\setminus\B)$, which implies that only finitely many members of the collection $\mathcal{A}=\Orb_G\big(\Jn(\nu)\cap(X\setminus\B)\big)$ intersect $C$. Setting $Z=(X\setminus\B)$ relative hyperbolicity of $G$ implies that $G$ acts cocompactly on $Z$. Applying Proposition \ref{proposition:lclfinaction} we get that $H$ acts cocompactly on $\Jn(\nu)\cap(X\setminus\B)$.
\end{proof}

\begin{lemma}
\label{lemma:cocompact on horoball}
Let $B\in\B$ be a horoball intersecting $\gJn(\nu)$. Then $S=\Stab_H\big(B)$ acts cocompactly on the horosphere $\Fr(B)\cap\gJn(\nu)$.
\end{lemma}

\begin{proof}
The family $\B$ of $G$-equivariant horoballs based at the parabolic points is locally finite. As $H$ acts cocompactly on $\Jn(\nu)\cap(X\setminus\B)$, we have that $H$ acts cocompactly on $\gJn(\nu)\cap(X\setminus\B)$. By Proposition \ref{proposition:lclfinaction} the group $S=\Stab_H(B)$ acts cocompactly on $\Fr(B)\cap\gJn(\nu)$.
\end{proof}

\begin{lemma}
\label{lemma:horoball based in nu}
Let $p\in\nu$ be a parabolic point of the action of $G$ on $X$. A horoball $B_p\in\B$ based at $p$ has unbounded intersection with $\gJn(\nu)$, the stabilizer $\Stab_H(B_p)$ is an infinite subgroup of $H$, and $p\in\Lambda(H)$.
\end{lemma}

\begin{proof}
Assume $p\in\nu$ is a bounded parabolic point, and let $B_p$ be the open horoball in $\B$ based at $p$. Because a necklace is the closure of a $\sim$-class, $p$ must be an accumulation point. So, we may find a sequence of  conical limit points $(x_n)$ which converge to $p$. As $\gJn(\nu)$ is a visibility space we find a sequence of geodesic lines $\big([p,x_n]\big)$ connecting $p$ to the elements of the sequence $(x_n)$. Define $y_n=[p,x_n]\cap \Fr(B_p)$. Then $(y_n)$ must be unbounded, 
because the sequence of end points $\big(\{p,x_n\}\big)$ converges to the pair $\{p,p\}$. Thus by Lemma \ref{lemma:conv seq equal GP to infty} the Gromov product $(p|x_i)_z\rightarrow\infty$ as $i\rightarrow\infty$.

By Lemma \ref{lemma:cocompact on horoball} $S=\Stab_H(B_p)$ acts cocompactly on $\Fr(B_p)\cap\gJn(\nu)$. Let $A$ be the fundamental domain for this action. By covering the sequence $(y_n)$ with the $S$-translates of $A$ we may find a sequence of group elements $(h_n')$ in $H$ so that $h_n\rightarrow p$.
\end{proof}

\begin{lemma}
\label{lemma:unifly bndd}
$\B^*$ be the collection of all horoballs for the action of $G$ which are based at parabolic points outside of $\nu$. The elements of $\B^*$ have uniformly bounded intersections with $\Jn(\nu)$.
\end{lemma}

\begin{proof}
The family of $G$-equivariant horoballs is locally finite, and $H$ acts properly and cocompactly on $\Jn(\nu)\cap(X\setminus\B)$. So, there are only finitely many $H$-orbits of horoballs in $\Jn(\nu)\cap(X\setminus\B)$. In particular, there are only finitely many $H$-orbits of horoballs based outside of $\nu$ which intersect $\Jn(\nu)$.

By Lemma \ref{lemma:cocompact on horoball} if $B_p$ is a horoball based a parabolic point $p\not\in\nu$ which intersects $\Jn(\nu)$, then $S=\Stab_H(B_p)$ acts cocompactly on $\Fr(B_p)\cap\gJn(\nu)$. This implies that if the intersection of $B_p$ with $\Jn(\nu)$ is infinite, then $p\in\Lambda(H)$. The limit set $\Lambda(H)$ is the minimal closed $H$-invariant subset of $\bndry X$. Thus $p\in\Lambda(H)$, implies $p\in\nu$, a contradiction. Therefore any horoball based outside of $\nu$ must have bounded intersection with $\Jn(\nu)$.
\end{proof}

Let $\B'\subset\B$ be the collection of horoballs based at parabolic points of $\nu$.

\begin{corollary}
\label{lemma:cocompact retricted truncation}
Let $H=\Stab_G(\nu)$. Then $H$ acts properly and cocompactly on $Y=\Jn(\nu)\cap(X\setminus\B')$.
\end{corollary}

\begin{proof}
By Lemma \ref{lemma:cocompact on truncation} gives us that $H$ acts cocompactly on $\Jn(\nu)\cap(X\setminus\B)$. Let $C$ be a fundamental domain for the action of $H$ on $\Jn(\nu)\cap(X\setminus\B)$. Since there only finitely many orbits of horoballs meeting $C$ there are only finitely many  orbits of horoballs from $B\setminus\B'$ which meet $C$, Lemma \ref{lemma:unifly bndd} implies that we may increase $C$ to a larger compact set $C'$ so that $H$ acts cocompactly on $Y$.
\end{proof}

Let $\bbp'\subset\bbp$ be $\bigset{P\in\bbp}{P=\Stab_G(B')\spc \text{for some}\spc B'\in\B'}$.

\begin{corollary}
\label{corollary:cusp uniform}
Let $\bbq=\bigset{H\cap P}{ P\in\bbp' }$. The action of $(H,\bbq)$ on $\gJn(\nu)$ is cusp uniform.
\end{corollary}

\begin{proof}
By definition each group in $\bbq$ stabilizes a horoball in $\B'$, and Lemma \ref{lemma:horoball based in nu} gives that every $Q\in\bbq$ is infinite. The collection $\B'$ is $H$-equivariant and Corollary \ref{lemma:cocompact retricted truncation} gives us that the action is cocompact on $Y=\Jn(\nu)\cap(X\setminus\B')$. Therefore the action of $(H,\bbq)$ on $\gJn(\nu)$ is cusp uniform.
\end{proof}

\begin{corollary}
$H$ acts geometrically finitely on $\nu$.
\end{corollary}

\begin{proof}
Corollary \ref{corollary:cusp uniform} shows that the action on $Y$ is cusp uniform, and Lemma \ref{lemma:actually the limit set} shows that $\nu=\bndry \gJn(\nu)$. Then Proposition \ref{proposition:cusp uniform is geom fin} the action on $\bndry \gJn(\nu)$ is geometrically finite.
\end{proof}

Bowditch \cite{Bow1} has shown that a geometrically finite action of a group on  metric space is always minimal. So we have:

\begin{corollary}
$H$ acts minimally on $\nu$
\end{corollary}

To complete the proof of Proposition \ref{prop:Key prop} it remains to show:

\begin{lemma}
The stabilizer of a jump in $\nu$ is a loxodromic and translate inseparable.
\end{lemma}

\begin{proof}
We first show that he stabilizer of a jump is loxodromic. Let $\Jump(\nu)$ be the set of jumps in $\nu$ and let $K'\subset Y$ be a compact set whose $H$-translates cover $Y$. In the hyperbolic setting Bowditch showed that $\Jump(\nu)/H$ is finite (see Lemma 5.19 of \cite{Bow98}). Bowditch's argument only used the fact that $G$ acts cocompactly on the space of unordered pairs in the boundary of $G$ and the fact that $\nu=\Lambda(H)$. Thus using Corollary \ref{corollary: cocompact on pair} and Lemma \ref{lemma:actually the limit set}, and the argument of Lemma 5.19 of \cite{Bow98} we may conclude that $\Jump(\nu)/H$ is finite. Thus $\bigset{\Jn(J)\cap Y}{J\in\Jump(\nu)}$ is locally finite in $Y$, and we may apply Proposition \ref{proposition:lclfinaction} to show that $\Stab_H(J)$ acts cocompactly on $\Jn(J)\cap Y$ for any jump $J\in\Jump(\nu)$.  Notice that if we knew that parabolic points did not participate in jumps, then we would know that for every jump $J$ there is a line in $\Jn(J)$ having bi-infinite intersection with $Y$. Then because action of $\Stab_H(J)$ on $\Jn(J)\cap Y$ is cocompact and by isometries we may extend the action of $\Stab_H(J)$ to $\bbr$, which implies that $\Stab_H(J)$ is loxodromic.

Let $S=\Stab_H(J)$. To see that we may extend the action of $S$ to $\bbr$, first let $\ell$ be line in $\Jn(J)$ and notice that if $C$ is a fundamental domain for the action of $S$ on $\ell$ then the convex hull $\Hull(C)$ of $C$ is a connected subset of $\ell$. Because the action of $S$ on $\ell$ is cocompact there is a bound on the diameter of components of $\ell\setminus\big\{\Orb_S(C)\big\}$. Additionally, there are at most two components of $\ell\setminus\big\{\Orb_S(C)\big\}$ which are adjacent to $\Hull(C)$. Let $\mathcal{V}$ be the collection of components of $\ell\setminus\big\{\Orb_S(C)\big\}$ whose intersection with $\Hull(C)$ is non-empty and define $C'$ to be the closure of $\Hull(C)\cap \mathcal{V}$. Then $C'$ is compact and $\ell\subset\Orb_S(C')$. As $\ell$ is isometric to $\bbr$ we are done.

We now show that parabolic points cannot participate in jumps. Let $J=\{x,y\}$ be a jump of $\nu$ and assume that $x$ was parabolic. Then $\Stab_H(x)$ cannot fix $y$. Let $h\in\Stab_H(x)$ be non-trivial. Then $h$ is a homeomorphism of $\bndry X$, so $\bigl\{x,h(y)\bigr\}$ is also a jump. Thus $x$ is a point participating in two jumps and must be isolated (see Section \ref{subsec:Cut Point Structures In Metric Spaces}), but by definition the parabolic points in a necklace cannot be isolated. Therefore $x$ could not have been parabolic.

Let $J=\{x,y\}$ be a jump in $\nu$. If $\N(x,y)\geq 3$, then $\{x,y\}$ is inseparable and thus $G$-translate inseparable. So, assume that $x\sim y$. To see that $\{x,y\}$ cannot be separated by a pair $\{a,b\}$ with $val(a)=val(b)=\N(a,b)=2$, notice that if such an $\{a,b\}$ Lemma \ref{lemma: separate implies cyclic} would imply that $x\sim y\sim a\sim b$ contradicting the fact that $\{x,y\}$ is a jump.

Assume $\{a,b\}$ is a cut pair with $\N(a,b)\geq 3$. Arguments of Bowditch \cite{Bow98} show that such a pair is an inseparable $M(3+)$ class (see specifically \cite{Bow98} Lemma 3.8 and Proposition 5.13). As $\{a,b\}$ is inseparable $\{a,b\}$ must lie in a single component of $M\setminus\{x,y\}$. Because $\{x,y\}$ is a cut pair there exists at least one component $U$ of $M\setminus\{x,y\}$ which does not contain $\{a,b\}$. Then $\overline{U}$ is a connected set contained in $M\setminus\{a,b\}$ such that $\{x,y\}\subset \overline{U}$. Thus $\{a,b\}$ cannot separate $\{x,y\}$.

As every loxodromic cut pair has two or more components in its complement, we have shown that $\{x,y\}$ is translate inseparable.
\end{proof}

\medskip
\subsection{Collecting local cut points}
\label{subsec:collecting local}

Now that we have proved Lemma \ref{lemma:control theorem} and Proposition \ref{prop:Key prop} we may plug into the arguments of Bowditch \cite{Bow98} in the case when $\relbndryG$ has no global cut points to describe the ways local cut points occur in $\relbndryG$.

\begin{theorem}
\label{theorem:collecting conical}
Let $(G,\bbp)$ be relatively hyperbolic and set $M=\relbndryG$. If $M$ is connected, has no global cut points, and is not homeomorphic to $S^1$, then we have the following:
\begin{enumerate}
\item A point $m\in M^*(2)$ is either in a necklace or a translate inseparable loxodromic cut pair.
\item $M^*(3+)$ consists of equivalence classes of translate inseparable loxodromic cut pairs.
\item A necklace $\nu$ in $M$ is homeomorphic to a Cantor set. Moreover, if $\nu$ is a Cantor set the jumps are loxodromic cut pairs, which are translate inseparable
\end{enumerate}
\end{theorem}

As first observed by  Guralnik \cite{Gur}, Lemma \ref{lemma:control theorem} allows us to apply Bowditch's arguments verbatim when considering only conical limit points to obtain Theorem \ref{theorem:collecting conical}(1) and (2). Theorem \ref{theorem:collecting conical}(3) also follows from the arguments of Bowditch by substituting Proposition \ref{prop:Key prop} for Bowditch's Lemma 5.17. We refer the reader to arguments of Section 5 of \cite{Bow98} for details.

As an immediate corollary we have:

\begin{corollary}
\label{corollary: non-cirlce necklace then lox cut pair}
Assume $\relbndryG$ is connected with no global cut points and not homeomorphic to $S^1$. If $\relbndryG$ has a non-parabolic local cut point, then $\relbndryG$ contains a $G$-translate inseparable loxodromic cut pair.
\end{corollary}

\medskip

We remark that Theorem \ref{theorem:collecting conical}(3) is related to the work of Groff (see Proposition 7.2 and the definition of relatively-QH in \cite{Grof13}). Also note that cut pairs are not separated by global cut points. Thus a necklace $\nu$ will be contained in some block of the form $\blockbndry$. This means we may now invoke the results of Section \ref{section:Reduction} to remove the hypothesis that $\relbndryG$ has global cut points and show:
\medskip

\begin{theorem}
\label{theorem:Collection thm}
Let $(G,\bbp)$ be a relatively hyperbolic group with tame peripherals and assume $\relbndryG$ is connected. If $p\in\relbndryG$ is a local cut point, then one of the following holds:
\begin{enumerate}
\item $p$ is parabolic point
\item $p$ is contained in a $G$-translate inseparable loxodromic cut pair
\item $p$ is in a necklace
\end{enumerate}
\end{theorem}

\begin{proof} Let $p$ be a local cut point. By Lemma \ref{label:ideal point} we have $p$ must be either a parabolic point or a conical limit point contained in a block. If $p$ is parabolic we are done, so assume that $p$ is a conical limit point. By Theorem \ref{theorem:treelike structure} the block is stabilized by a subgroup $H$, and $H$ is hyperbolic relative to $\bbq$. Theorem \ref{theorem:treelike structure} also implies that $\blockbndry$ is connected and has no global cut points, if $\blockbndry$ is not a circle, we may apply Theorem \ref{theorem:collecting conical} to $\blockbndry$. Thus $\blockbndry$ contains a necklace or an inseparable loxodromic cut pair which contains $p$. Corollary \ref{corollary:cut pair in block} implies that inseparable loxodromic cut pairs and necklaces in $\blockbndry$ correspond to inseparable loxodromic cut pairs and necklaces in $\relbndryG$.
If $\blockbndry$ is a circle, then $\blockbndry$ is a necklace containing $p$, and we are done by Corollary \ref{corollary:cut pair in block}.\end{proof}


\section{Splitting theorem}
\label{section:splitting thm}

Throughout this section we will assume that $(G,\bbp)$ is relatively hyperbolic with tame peripherals and that $\relbndryG$ is connected. Having developed the appropriate tools in Sections \ref{section:Reduction} and \ref{section:Local Cut Points in relbndry}, we now wish to prove Theorem \ref{theorem:Splitting Theorem}. We start with a few lemmas.

\begin{lemma}
\label{lemma:component a circle}
Assume that $\relbndryG$ is not homeomorphic to a circle. If $H$ is the stabilizer of a block and $\compbndry$ is homeomorphic to a circle, then there exists a non-trivial peripheral splitting of $G$ over a $2$-ended subgroup.
\end{lemma}

\begin{proof} If $\compbndry$ is a circle, then Theorem \ref{theorem:treelike structure}(4) and a result of Tukia (see \cite{T88} Theorem 6B) imply that $H$ is virtually a surface group, and the peripheral subgroups are the boundary subgroups of that surface. Because $\relbndryG$ is not a circle, there is a global cut point $p$ of $\relbndryG$ contained in $\blockbndry$ such that $\stab_H(p)$ is a $2$-ended subgroup. As the boundary is connected, Corollary \ref{corollary:theorem :non-trivial splitting} implies that $G$ must split over an infinite subgroup of $\stab_H(p)$.\end{proof}

\begin{lemma}
\label{lemma: quotient of loxs}
Let $\{a,b\}$ be a translate inseparable cut pair in $\relbndryG$ and $Q$ the quotient space obtained by identifying $ga$ to $gb$ for every $g\in G$. Then $Q$ contains a cut point for each pair in $\Orb_G\big(\{a,b\}\big)$.
\end{lemma}

\begin{proof} Let $M=\relbndryG$ and assume $\{a,b\}$ is a translate inseparable cut pair. As the action of $G$ is by homeomorphisms $\{ga,gb\}$ is inseparable for every $g\in G$.

Define $q\colon\relbndryG\rightarrow Q$ to be the quotient map described in the statement of the lemma. Let $\C$ be the collection of components of $\relbndryG\setminus\{c,d\}$ for some pair $\{c,d\}$ in $\Orb_G\big(\{a,b\}\big)$. Because every pair in $\Orb_G(\{a,b\})$ is inseparable, there does not exist a pair $\{x,y\}\in\Orb_G(\{a,b\})$ which meets two elements of $\C$. Thus if $C_1$ and $C_2$ are distinct components of $\relbndry\setminus\{c,d\}$, we have that $q(C_1)$ and $q(C_2)$ are disjoint connected components of $Q\setminus\big\{q(c)=q(d)\big\}$.
\end{proof}

By Corollary 1.7 of \cite{Osin04} we have:

\begin{lemma}
\label{lemma:Osin lemma}
Let $(G,\bbp)$ be relatively hyperbolic. Assume that $g$ is a loxodromic element contained in a maximal $2$-ended subgroup $H$. By adding $H$ and all of its conjugates to $\bbp$, we may extend $\bbp$ to a new peripheral structure $\bbp'$ so that $(G,\bbp')$ is relatively hyperbolic.
\end{lemma}

Let $(G,\bbp)$ and $(G,\bbp')$ be as in Lemma \ref{lemma:Osin lemma}. We say that $(G,\bbp')$ is the {\it loxodromic extension of $(G,\bbp)$ by $g$}.

\begin{lemma}
\label{lemma:Q homeo to extended bndry}
Assume $\relbndryG$ contains an inseparable loxodromic cut pair $\{a,b\}$ stabilized by a loxodromic element $g$, and let $Q$ the quotient space obtained by identifying $g'a$ to $g'b$ for every $g'\in G$. If $(G,\bbp')$ is the loxodromic extension of $(G,\bbp)$ by $g$, then $Q$ is equivariantly homeomorphic to $\partial(G,\bbp')$.
\end{lemma}

Lemma \ref{lemma:Q homeo to extended bndry} was proved by Dahmani \cite{Dahm03} in the case where $\langle g\rangle$ is a maximal $2$-ended subgroup and follows from Lemma 4.16 of \cite{Y14} in the general case.

\begin{lemma}
\label{lemma:cut pair implies splitting}
Assume $\relbndryG$ contains a translate inseparable loxodromic cut pair. Then $G$ splits relative to $\bbp$ over a two-ended group.
\end{lemma}

\begin{proof} Assume the hypothesis. Then there is a loxodromic group element $g\in G$  which stabilizes the loxodromic cut pair, and  such that $\langle g \rangle$ is contained in a maximal $2$-ended subgroup $H$. Let $(G,\bbp')$ be the loxodromic extension of $(G,\bbp)$ by $g$. By Lemma \ref{lemma: quotient of loxs} and Lemma \ref{lemma:Q homeo to extended bndry} there is a cut point $\bndry(G,\bbp')$ stabilized by $H$, which by Corollary \ref{corollary:theorem :non-trivial splitting} implies that $(G,\bbp')$ has a non-trivial peripheral splitting over a subgroup of $H$. As $\partial(G,\bbp')$ is connected $G$ does not split over a finite group relative to $\bbp$. Since every infinite subgroup of $H$ is $2$-ended, we are done. \end{proof}

Lastly, to prove Theorem \ref{theorem:Splitting Theorem} we will use the following lemma taken from the first paragraph in the proof of Theorem 7.8 of \cite{GM16}. Lemma \ref{lemma: Groves Manning} below is more general than what is stated in \cite{GM16}, but follows directly from Groves and Manning's proof, which we include for completeness. The proof of Lemma \ref{lemma: Groves Manning} uses the cusped space for $(G,\bbp),$ and we refer the reader to \cite{GM08} Section 3 for the construction of the cusped space. 

\begin{lemma}
\label{lemma: Groves Manning}
Let $(G,\bbp)$ be relatively hyperbolic with tame peripherals. Assume that $\relbndry$ is connected and not homeomorphic to a circle. if $G$ splits over a non-parabolic $2$-ended subgroup relative to $\bbp$, then $\relbndryG$ contains a non-parabolic local cut point.
\end{lemma}

\begin{proof}
Assume the hypotheses, and let $H$ be a non-parabolic $2$-ended subgroup over which $G$ splits relative to $\bbp$. Because $H$ is non-parabolic, $H$ quasi-isometrically embeds in the cusped space $X(G,\bbp)$. Since this splitting is relative to $\bbp$ the cusped space $X(G,\bbp)$ can be realized as a tree of cusped spaces glued together in the pattern of the Bass-Serre tree for the splitting over $H$. Thus $H$ coarsely separates $X(G,\bbp)$ into at least two components, and the limit set of $H$ is a pair of non-parabolic local cut points which separate $\relbndryG$.
\end{proof}

\begin{proof}[Proof of Theorem \ref{theorem:Splitting Theorem}]
By Lemma \ref{lemma: Groves Manning} if $G$ splits over a non-parabolic $2$-ended subgroup relative to $\bbp$, then $\relbndry$ contains a non-parabolic local cut point.


Now, assume that $x\in\relbndryG$ is a non-parabolic local cut point. By Theorem \ref{theorem:Collection thm} we know that $x$ is contained in either a translate inseparable loxodromic cut pair or a necklace. If $x$ is in a translate inseparable loxodromic cut pair we are done by Lemma \ref{lemma:cut pair implies splitting}.

Assume $x$ is in a necklace $\nu$. Then $\nu$ is either a circle or it is not. If $\nu$ is homeomorphic to $S^1$ we are done by Lemma \ref{lemma:component a circle}. If $\nu$ is not a circle, then $\nu$ contains a translate inseparable loxodromic cut pair by Lemma \ref{theorem:collecting conical}, and again we are done by Lemma \ref{lemma:cut pair implies splitting} and Corollary \ref{corollary:cut pair in block}.
\end{proof}

\section{Ends of generalized Peano continua admitting proper and cocompact group actions}
\label{section:Ends of GenPCont}

Let $G$ be a group with finite generating set $S$, and let $\caygs$ denote the Cayley graph of $(G,S)$. We define $\Ends(G)$ to be $\Ends(X)$ for any generalized Peano continuum $X$ on which $G$ acts properly and cocompactly. The goal of this section is to prove that $\Ends(G)$ is well defined (see Theorem \ref{theorem:ends theorem}). In particular, we show that $\Ends\big(\caygs\big)$ is homeomorphic to $\Ends(X)$ for all generalized Peano continua $X$. A special case of this result is well known for groups acting on CW-complexes (see for example \cite{Geo}). Theorem \ref{theorem:ends theorem} provides a generalization to generalized Peano continua, a classes of spaces which need not be CW-complexes. We remark that the techniques used to prove Theorem \ref{theorem:ends theorem} differ from those found in \cite{Geo}.

\medskip
One consequence of Theorem \ref{theorem:ends theorem} for the boundary $\relbndryG$ of a relatively hyperbolic group $G$ is that if the peripherals are one-ended then a parabolic point can be a local cut point if and only if it is a global cut point (see Corollary \ref{corrolary: p not local cut point}). This particular fact will be required for the proof of Theorem \ref{theorem: Classification thm}.

Let $G$ be a finitely generated discrete group acting properly and cocompactly on a generalized Peano continuum $X$. We will use Proposition \ref{proposition:Craig's prop} to prove Theorem \ref{theorem:ends theorem}. To begin the proof first construct a proper map $\Phi \colon\caygs\rightarrow X$ from the Cayley graph of $G$ to $S$ in the following way:

Let $S$ be a finite generating set for the group $G$. Fix a base point $x_0$ in the fundamental domain of the action of $G$ on $X$, and for every vertex $v_g$ in $\caygs$ define $\Phi(v_g)=g.x_0$. For every $s\in S\cup S^{-1}$ fix a path $p_s$ in $X$ with $p_s(0)=x_0$ and $p_s(1)=s.x_0$. We will denote $P(S)$ the collection of paths found in this way, i.e. $P(S)=\big\{p_s|s\in S\big\}$. Now, for any edge $e_s\in\caygs$ with end points $v_g$ and $v_{gs}$ define $\Phi(e_s)$ to be $gp_s$. Notice that $\Phi$ is well defined because $gp_s$ is a path with end points $gx_0$ and $gsx_0$ for every $g$ and $s$. Also, note that by the pasting lemma $\Phi$ is continuous.

\medskip
\begin{lemma}
The map $\Phi \colon\caygs\rightarrow X$ is proper for all $S$.
\end{lemma}

\medskip
\begin{proof}  Let $A\subseteq X$ be compact. As $X$ is Hausdorff $A$ is closed, therefore $\Phi^{-1}(A)$ is closed.
We show that $\Phi^{-1}(A)$ intersects only finitely many vertices and edges. Assume that $\Phi^{-1}(A)$ meets infinitely vertices. This implies that $A$ contains $g_nx_0$ for infinitely many $g_n\in G$, contradicting properness of the action of $G$ on $X$.

 Now assume that infinitely many edges meet $\Phi^{-1}(A)$. As there are finitely many orbits of edges there must be infinitely many edges with the same label, say $s$, meeting $\Phi^{-1}(A)$. Thus we may find an infinite sequence of group elements, $(g_i)_{i=1}^{\infty}$ such that $g_ip_s\cap A\neq \emptyset$ for every $i$. Set $C=p_s\cup A$, then $C$ is compact and $C\cap g_iC\neq\emptyset$ for every $i$, again a contradiction.\end{proof}

\medskip
Define $\Phi^*\colon \Ends\big(\caygs\big)\rightarrow \Ends(X)$ be the ends map induced by $\Phi$.

\begin{lemma}
\label{lemma:ends surjection}
$\Phi^*$ is a 
surjection for all $S$.
\end{lemma}

\medskip
\begin{proof}  Let $K\subset X$ be a compact connected set whose $G$-translates cover $X$, let $\{C_i\}_{i=1}^{\infty}$ be an exhaustion of $X$, and let $E=(E_1,E_2,E_3,...)\in \Ends(X)$.

Let $x_i\in E_i$ for some $i$. The translates of $K$ cover $X$, so there exists some $g_i\in G$ such that $x_i\in g_iK$. As $g_iK$ is compact there exists some $j\in\bbn$ such that $gK\subseteq C_j$. Let $x_j\in E_j\subset X\backslash C_j$ as before there exists some $g_j\in G$ such that $x_j\in g_jK$ and some $C_k$ containing $g_jK$. So we may pass to a subsequence $(E_{i_1},E_{i_2},E_{i_3},...)$ of $E$ corresponding to a sequence of distinct group elements $(g_{i_1},g_{i_2},g_{i_3},...)$ of $G$ found in the manner just described.

The sequence  $(g_{i_1},g_{i_2},g_{i_3},...)$ corresponds to an infinite sequence, $(v_{g_{i_j}})_{j=1}^{\infty}$, of distinct vertices in $\caygs$. Because the map $\Phi$ is proper, compactness of $\caygs\cup\Ends(\caygs)$ we have that some subsequence $(v_{g_{i_{j_k}}})_{k=1}^{\infty}$of $(v_{g_{i_j}})_{j=1}^{\infty}$ must converge to an end of $\caygs$. Using the path connectedness of $E_i$ we may find a proper ray, $r$, in $\caygs$ containing the vertices $(v_{g_{i_{j_k}}})_{k=1}^{\infty}$. The ray $r$ determines an end of $\caygs$, which by construction $\Phi$ maps to the end $E$ under $\Phi^*$. Thus $\Phi^*$ is surjective.\end{proof}

To complete the proof of Theorem \ref{theorem:ends theorem} we will need following well known result about ends of Cayley graphs \cite{SW79}, which is a special case of Theorem \ref{theorem:ends theorem}.

\begin{theorem}
\label{theorem: change gen set}
Assume $G$ is finitely generated, and let $S$ and $T$ be two finite generating sets for $G$. Then $\Ends\big(\caygs\big)$ is homeomorphic to $\Ends\big(\caygt\big)$.
\end{theorem}

\medskip
We will also require the following lemma:

\begin{lemma}
\label{lemma:connected compact set}
Let $G$ act properly and cocompactly on a generalized Peano continuum $X$. Then there exists a connected compact set $K$ whose $G$-translates cover $X$.
\end{lemma}

\begin{proof}
The proof is similar to that of Lemma 9.6 of \cite{HR1}. Let $C$ be a compact set whose $G$-translates cover $X$. Let $x\in C$. By local compactness we have that $x$ has a compact neighborhood $U$, and local connectedness implies that the interior $\Int(U)$ contains a connected neighborhood $V$ of $x$. The closure $\overline{V}$ of $V$ is a compact connected neighborhood of $x$. As $C$ is compact we may cover $C$ by finitely many such neighborhoods. The union of these neighborhoods $K'$ is compact and consists of finitely many components. As $X$ is arcwise connected, we may attach finitely many arcs to $K'$ to find a compact connected set $K$ such that $C\subset K'\subset K$.
\end{proof}

\begin{proof}[Proof of Theorem \ref{theorem:ends theorem}]
By Lemma \ref{lemma:connected compact set} there exists a connected compact set $K$ whose $G$ translates cover $X$. We will also assume that $K$ contains the base point $x_0$. Define $S$ to be $\bigset{s\in G}{K\cap sK\neq\emptyset}$. It is a standard result that $S$ generates $G$. By Theorem \ref{theorem: change gen set} $\Ends\big(\caygs\big)$ is independent of choice of generating set.

Let $\Phi\colon\caygs\rightarrow X$ be the map defined at the beginning of this section with $\Phi(1_G)=x_0$. By Lemma \ref{lemma:ends surjection} we need only show that $\Phi^*$ is injective. To do this we will make use of Lemma \ref{lemma: ladder equivalence}.

Let $\alpha$ and $\beta$ be proper rays in $\caygs$ and $(a_i)$ and $(b_i)$ the corresponding sequences of vertices. Note that, if necessary, $\alpha$ and $\beta$ may be homotoped to combinatorial proper rays, so we may assume that no vertex in $(a_i)$ or $(b_i)$  occurs infinitely many times. Assume that $\Phi(\alpha)$ and $\Phi(\beta)$ are in the same ladder equivalence class. Then we may find a proper map of the infinite ladder into $X$ such that $\Phi(\alpha)$ and $\Phi(\beta)$ form the sides; moreover, by concatenating paths if necessary we may assume that the rungs, $r_i$, of the ladder have end points $\Phi(a_i)$ and $\Phi(b_i)$. Call this ladder $L$. Note that the rungs $r_i$ of $L$ may not pull back to paths in $\caygs$ under $\Phi^{-1}$. We show that we can find an alternate sequence of rungs $\rho_i$ connecting $\Phi(a_i)$ to $\Phi(b_i)$ and such that each $\rho_i$ pulls back to an edge path in $\caygs$.

For any rung $r_i$ we may find a finite number of translates of $K$ that cover $r_i$. Let $\{g_1,g_2,...,g_n\}$ be such that $\im(r_i)\subset \bigcup_{j=1}^{n}g_j K$. Notice that by connectedness of the rung $r_i$ we may assume that $\{g_1,g_2,...,g_n\}$ is enumerated in such a way that $g_jK\cap g_{j+1}K\neq\emptyset$. Consequently, the $g_jK$ form a chain of connected compact neighborhoods such that the points $g_i x_0$ in the translates of $K$ can be connected by paths which are translates of paths in $P(S)$ (see the construction of $\Phi$); in other words, because of the specific choice of generating set they are the images of edges in $\caygs$. By concatenating paths in $\Orb_G(P(S))$ we may find a path $\rho_i$ which pulls back to an edge path in $\caygs$ connecting $(a_i)$ and $(b_i)$.

Lastly, we need to check that some sub-ladder of the ladder $L$ pulls back to a ladder in $\caygs$ under $\Phi$. Let $C\subset\caygs$ be a compact. We find a $\rho_i$ such that $\Phi^{-1}(\rho_i)$ is in $\caygs\setminus C$.

Set $C'=\Phi(C)$ and $K'=\big(\bigcup\limits_{s\in S} sK\big)\cup P(S)$. Assume that there does not exist a subsequence of rungs $\{\rho_i\}$ entirely outside of $C'$. Then we may find a compact set $N=\bigcup\limits_{g\in I} gK'$ where $I=\set{g\in G}{K'\cap gK'\neq\emptyset}$ such that every rung, $r_i$ of $L$ meets $N$. As the ladder $L$ was proper this is a contradiction. Thus there must exist a $\rho_i$ outside of $\Phi(C)$, which implies that $\Phi^{-1}(\rho_i)\subset\caygs\setminus C$. Therefore as $C$ was chosen to be arbitrary we have that $\alpha$ and $\beta$ represent the same end of $\caygs$.\end{proof}

As an immediate corollary we obtain:

\begin{corollary}
\label{corollary:one ended}
Let $G$ be a one-ended finitely generated group acting properly and cocompactly on a generalized Peano continuum $X$. Then $X$ is $1$-ended.
\end{corollary}

In particular, we have:

\begin{corollary}
\label{corrolary: p not local cut point}
 Let $(G, \bbp)$ be relatively hyperbolic with tame peripherals and every $P\in\bbp$ $1$-ended. If $p$ is parabolic point in $\relbndryG$ which is not a global cut point, then $p$ cannot be a local cut point.
\end{corollary}

\begin{proof} Assume the hypotheses and let $P$ be the maximal parabolic subgroup which stabilizes $p$. Bowditch \cite{Bow1} has shown that $P$ acts properly and cocompactly on $\relbndryG\setminus\{p\}$. Because $p$ is not a global cut point, we know that $\relbndryG\setminus\{p\}$ is connected. We are assuming that $(G,\bbp)$ has tame peripherals, so $\relbndryG$ is locally connected. Thus, $\relbndryG\setminus\{p\}$ is an open connected subset of a Peano continuum; consequently, $\relbndryG\setminus\{p\}$ is a generalized Peano continuum and we may apply Corollary \ref{corollary:one ended}. \end{proof}


\section{Classification theorem}

\label{section:classification}

In this section we prove Theorem \ref{theorem: Classification thm}. This theorem is a generalization of a theorem due to Kapovich and Kleiner \cite{KK00} concerning the boundaries of hyperbolic groups. Kapovich and Kleiner's proof used the topological characterizations of the Menger curve \cite{A58a, A58b} and Sierpinski carpet \cite{Whyburn}. A compact metric space $M$ is a Menger curve provided: M is $1$-dimensional, M is connected, M is locally connected, M has no local cut points, and no non-empty open subset of M is planar. If the last condition is replaced with, ``M is planar,'' then we have the topological characterization of the Sierpinski carpet. Having proved Theorem \ref{theorem:Splitting Theorem} and Theorem \ref{theorem:ends theorem}, the only remaining step in the proof of Theorem \ref{theorem: Classification thm} is the following argument inspired by Kapovich and Kleiner \cite{KK00}.

\begin{proof}[Proof of Theorem \ref{theorem: Classification thm}] Assume the hypotheses and assume that $\relbndryG$ is not homeomorphic to a circle. Then $\relbndryG$ is a compact and 1-dimensional metric space. Because $G$ is one-ended, $\relbndryG$ is connected. Since we are assuming $(G,\bbp)$ has tame peripherals, connectedness of $\relbndryG$ implies that it must also be locally connected (see Theorem \ref{theorem:Bow Local conn}).

There are two types of local cut points, those that separate $\relbndryG$ globally and those that  do not. By Theorem \ref{theorem:global cut point implies splitting} the no peripheral splitting hypothesis implies that $\relbndryG$ is without global cut points. Additionally, the peripheral subgroups are assumed to be $1$-ended, so by Theorem \ref{theorem:ends theorem} there are no parabolic local cut points. Thus any local cut point must be a conical limit point. If there were a conical limit local cut point, then Theorem \ref{theorem:Splitting Theorem} would imply that $G$ splits over a $2$-ended subgroup, a contradiction.

Now, $\relbndryG$ is planar or it is not. If it is planar then it is a Sierpinski carpet. Assume $\relbndryG$ is not planar, then by the Claytor embedding theorem \cite{Cla} it must contain a  topological embedding of a non-planar graph $K$. It suffices to find a homeomorphic copy of $K$ inside any open neighborhood $V$ in $\relbndryG$.

As conical limit points are dense, let $x$ be a conical limit point in $\relbndryG\setminus\{K\}$. By definition of conical limit point there exists $a,b\in\relbndryG$ and a sequence of group elements $(g_i)\subset G$ such that $g_i x\rightarrow a$ and $g_i z\rightarrow b\neq a$ for every $z\in\relbndryG\setminus\{x\}$. Now, $G$ acts on $\relbndryG$ as a convergence group. Thus we have that the sequence$(g_i)$ restricted to $\relbndryG\setminus\{x\}$ converges locally uniformly to $b$, so we may find a homeomorphic copy of $K$ inside any neighborhood $U$ of $b$.

Let $V$ be any neighborhood in $\relbndryG$. The action of $G$ on $\relbndryG$ is minimal (see \cite{Bow1}), so we have that there exists some group element $g$ such that $gb\in V$. Let $W$ be a neighborhood of $gb$ inside $V$ and set $U$ from the previous paragraph equal to $g^{-1}(W)$. Then we may find a homeomorphic copy of $K$ inside of $V$. \end{proof}

\bibliographystyle{plain}
\bibliography{mybib}{}

\end{document}